%% file: ci.tex
\documentclass[twoside]{birkmult}

\usepackage{amsmath, amsthm, amssymb, mathrsfs}
\usepackage{epsfig, graphicx}

\newtheorem{thm}{Theorem}
\newtheorem{cor}[thm]{Corollary}

\newtheorem*{df}{Definition}

\numberwithin{equation}{section}

\include{notation}

\begin{document}

\title[Convergent Interpolation to Cauchy Integrals over Analytic Arcs]{Convergent Interpolation to Cauchy Integrals over Analytic Arcs}

\author[L. Baratchart]{Laurent Baratchart}

\address{INRIA, Project APICS \\
2004 route des Lucioles --- BP 93 \\
06902 Sophia-Antipolis, France}

\email{laurent.baratchart@sophia.inria.fr}

\author[M. Yattselev]{Maxim Yattselev}

\address{Corresponding author, \\
INRIA, Project APICS \\
2004 route des Lucioles --- BP 93 \\
06902 Sophia-Antipolis, France}

\email{myattsel@sophia.inria.fr}

\date{\normalsize \today}

\begin{abstract}
We consider multipoint Pad\'e approximation to Cauchy transforms of complex measures. We show that if the support of a measure is an analytic Jordan arc and if the measure itself is absolutely continuous with respect to the equilibrium distribution of that arc with Dini-smooth non-vanishing density, then the diagonal multipoint Pad\'e approximants associated with appropriate interpolation schemes converge locally uniformly to the approximated Cauchy transform in the complement of the arc. This asymptotic behavior of Pad\'e approximants is deduced from the analysis of underlying non-Hermitian orthogonal polynomials, for which we use classical properties of Hankel and Toeplitz operators on smooth curves. A construction of the appropriate interpolation schemes is explicit granted the parametrization of the arc.
\end{abstract}



\maketitle

\section{Introduction}
\label{sec:intro}

Rational approximation to analytic functions of one complex variable is a most classical subject which has undergone many developments since C. Runge's proof that such an approximation is indeed possible, locally uniformly on the domain of holomorphy \cite{Run85}. Let us quote for example the deep study of those open sets for which holomorphic functions that extend continuously up to the boundary can be approximated by rational functions on the closure \cite{Mer62,Vit66}, see comprehensive expositions and further references in the monograph \cite{Gamelin}. In another connection, the achievable rate of convergence of rational approximants at regular points, when the degree goes large, also received considerable attention \cite{Walsh, Gon78c, LevLub86, Par86, Pr93}. Meantime, rational approximation has become a tool in numerical analysis \cite{Lev73, Wen89, TrWeidSchmel06}, as well as in the modelling and control of signals and  systems, see for instance \cite{Antoulas, CameronKudsiaMansour, B_CMFT99, Gl84, Partington2}.

From the constructive viewpoint, a great deal of interest has been directed to those rational functions of type\footnote{A rational function is said to be of type $(m,n)$ if it can be written as the ratio of a polynomial of degree at most $m$ and a polynomial of degree at most $n$.} $(m,n)$ that interpolate a given function in $m+n+1$ points, counting multiplicity. These are the so-called {\em multipoint Pad\'e approximants} \cite{BakerGravesMorris}, which subsume the {\em classical Pad\'e approximants} that interpolate the function in a single point with multiplicity $m+n+1$ \cite{Pade92}. Beyond the natural appeal of such an elementary {\em linear} procedure, reasons why Pad\'e approximants have received continuing attention include their early number theoretic success when applied to certain entire functions \cite{Herm73, Siegel, Skor03}, the remarkable behavior of diagonal ({\it i.e.} of $(n,n)$-type) Pad\'e approximants  to Markov functions \cite{Mar95, GL78}, the de Montessus de Ballore theorem and its generalizations on the convergence of $(m,n)$-type Pad\'e approximants for fixed $n$ \cite{S72} as well as the solution to the ``disk problem'' for diagonal sequences \cite{Gon82} that give rise to {\em extrapolation} schemes for analytic functions, the Nuttall-Pommerenke theorem on convergence in capacity of diagonal Pad\'e approximants to functions having singular set of zero capacity \cite{Nut70, Pom73}, the numerical use of Pad\'e approximants for initial value problems and convergence acceleration \cite{But64, Gr65, IserlesNorsett, BrezinskiRedivoZaglia}, their connections to quantum mechanics and quantum field perturbation theory \cite{Baker, Tj_PRA77}, and the research impetus generated by the so-called Pad\'e conjecture on the convergence of a subsequence of the diagonal sequence \cite{Bak73}, in the largest disk of meromorphy, that was eventually settled in the negative \cite{Lub03}. The reader will find a detailed introduction to most of these topics, as well as further references, in the comprehensive monograph \cite{BakerGravesMorris}.

In the present paper, which deals with the convergence of multipoint Pad\'e approximants to functions defined as Cauchy integrals over a compact arc, the relevant approximants are the diagonal ones with one interpolation condition at infinity for they are those producing the (generically simple) zero assumed by the function at infinity. The first class of Cauchy integrals for which diagonal Pad\'e approximants were proved convergent is the class of Markov functions, that is, Cauchy transforms of positive measures on a real segment \cite{Mar95,GL78}. In fact, under mild conditions, the diagonal multipoint Pad\'e approximants to such functions at a conjugate-symmetric system of interpolation points converge strongly ({\it i.e.}  locally uniformly in the domain of analyticity), as the number of interpolation conditions goes large. The working pin of this result is the intimate connection between Pad\'e approximants and orthogonal polynomials: the denominator of the $n$-th multipoint  Pad\'e approximant is the $n$-th orthogonal polynomial with respect to the measure defining the Markov function, weighted by the inverse of the polynomial whose zeros are the interpolation points (this polynomial is identically 1 for classical Pad\'e approximants). This is a key to sharp convergence rates, using the asymptotic theory of orthogonal polynomials with varying weights \cite{GL78, StahlTotik}.

This result has been generalized in several directions, in various attempts to develop a convergence theory of rational interpolants to more general Cauchy integrals. In particular, the strong convergence of classical diagonal Pad\'e approximants to Cauchy transforms of {\em complex-valued} functions on a segment was established in \cite{Bax61, Nut70, NS80}, when the density of the function with respect to the arcsine distribution on the segment is smoothly invertible, and in \cite{Mag87} under the mere assumption that the function has continuous argument and vanishes almost nowhere. It is interesting that the first three references proceed {\it via} strong asymptotics for {\em non-Hermitian} orthogonal polynomials on a segment, whereas the last one relies on different, operator theoretic methods. Multipoint Pad\'e approximants to Cauchy transforms of functions with non-vanishing {\em analytic} density with respect to the arcsine distribution on a real segment were in turn proved strongly convergent under quantitative assumptions on the near conjugate-symmetry of the interpolation scheme \cite{AVA04, Ap02}; the proof rests on strong asymptotics for non-Hermitian orthogonal polynomials with varying weight that dwell on the Riemann-Hilbert approach to classical orthogonal polynomials \cite{Deift} as adapted to the segment in \cite{KMLVAV04}. Let us mention that the occurrence of zeros in the density can harm the convergence, because some poles of the approximant may no longer cluster to the segment in this case, but if the density has an argument of bounded variation the number of these ``spurious'' poles remains bounded and one can still get convergence {\em in capacity} even for substantial zeroing of the density \cite{uBY2}.

The case of Cauchy integrals over more general arcs than a segment turns out to be harder. In the series of pathbreaking papers \cite{St85, St86, St89, St97}, classical diagonal Pad\'e approximants to functions with branchpoints were shown to converge {\em in capacity} on the complement of the system of arcs of minimal logarithmic capacity outside of which the function is analytic and single-valued. This extremal system of arcs, often called  nowadays a {\it symmetric contour} or an {\em $S$-contour}, is characterized by a symmetry property of the (two-sided) normal derivatives of its equilibrium potential, and the above-mentioned convergence ultimately depends on a deep potential-theoretic analysis of the zeros of non-Hermitian orthogonal polynomials over $S$-contours. Shortly after, the result was extended to multipoint Pad\'e approximants for Cauchy integrals of continuous quasi-everywhere non-vanishing functions over $S$-contours minimizing some {\em weighted capacity}, provided that the interpolation points asymptotically distribute like a measure whose potential is the logarithm of the weight \cite{GRakh87}. With these works it became transparent that the appropriate Cauchy integrals for Pad\'e approximation must be taken over $S$-contours, and that the interpolation points should distribute according to the weight that defines the symmetry property. Subsequently, {\em strong} convergence of diagonal multipoint Pad\'e approximants was obtained in this setting for analytic $S$-contours and non-vanishing {\em analytic} densities with respect to either the equilibrium distribution of the contour or the arclenth, under certain quantitative assumptions on the convergence of the interpolation points \cite{Ap02}. Earlier results in a specific case can be found in \cite{Suet00}.

Surprisingly perhaps, the natural inverse problem whether {\em given} a system of arcs, say $\mathcal{S}$, {\em there exists} a logarithmic potential whose exponential defines a suitable weight making $\mathcal{S}$ into an $S$-contour, and whether the interpolation points can be chosen accordingly to produce strong convergence of multipoint Pad\'e approximants to Cauchy integrals on $\mathcal{S}$, was apparently not considered. The goal of the present paper is to fill this gap when $\mathcal{S}$ is smooth. More precisely, we prove that a smooth arc is an $S$-contour for a weight whose logarithm is the potential of a positive measure supported disjointly from the arc if, and only if the arc is analytic ({\it cf.} Theorem \ref{thm:sp}). When this is the case, we further show there exists a scheme of interpolation points producing strong convergence of the sub-diagonal multipoint Pad\'e approximants to Cauchy integrals on the arc, provided the density of the integral with respect to {\em some} positive power of the equilibrium measure of the arc is Dini-smooth and non-vanishing ({\it cf.} Theorems \ref{thm:sa1} and \ref{thm:pade}). This result seems first to ascertain convergence of an interpolation scheme for Cauchy integrals of fairly general densities (note analyticity is not required) over a rather large class of arcs (namely analytic ones). Finally, we show that the non-vanishing requirement on the density can be relaxed a little ({\it cf.} Theorems \ref{thm:sa2} and \ref{thm:pade}), thereby providing an initial example, over an arbitrary analytic arc, of a family of functions whose zeroing at isolated places does not destroy the strong convergence of multipoint Pad\'e approximants to their Cauchy integral.

Our proofs differ somewhat from those commonly encountered in the field. In particular, we translate the symmetry property, which is of geometric character, into an analytic one, namely the existence of a sequence of pseudo-rational functions tending to zero off the arc and remaining bounded on the latter. Choosing the interpolation points to be the zeros of these pseudo-rational functions, we transform the non-Hermitian orthogonality equation satisfied by the denominator of the multipoint Pad\'e approximant into an integral equation involving Hankel and Toeplitz operator on the arc $\mathcal{S}$. Then, the bounded behavior of our pseudo-rational functions teams up with classical compactness properties of Hankel operators to produce strong asymptotics for the non-Hermitian orthogonal polynomials under study. This in turn provides us with locally uniform error rates for the approximant on the complement of $\mathcal{S}$. From the technical point of view the paper can be seen as deriving strong asymptotics for non-Hermitian orthogonal polynomials with complex densities with respect to the equilibrium distribution of the arc under minimum smoothness assumptions as compared to those currently used in the Riemann-Hilbert methods \cite{uMcLM,Ap02}.

The paper is organized as follows. In the next section we formulate our main results (Theorems \ref{thm:sp}--\ref{thm:pade}). Section \ref{sec:aux} contains the necessary background material for the proofs adduced in Section \ref{sec:proofs}. The last section contains an example of a class of contours for which the behavior of orthogonal polynomials is ``highly similar'' to the case when $F=[-1,1]$.

\section{Main Results}

Let $F$ be a rectifiable Jordan arc with endpoints $\pm1$ oriented from $-1$ to $1$. Set
\begin{equation}
\label{eq:chebyshev}
w(z) = w(F;z):= \sqrt{z^2-1}, \quad w(z)/z \to 1 \quad \mbox{as} \quad z\to\infty,
\end{equation}
which is a holomorphic function outside of $F$. Then $w$ has continuous traces (boundary values) from each side of $F$, denoted by $w^+$ and $w^-$ (e.g. $w^+$ is the trace of $w$ taken from the left side of $F$ as $F$ is traversed in the positive direction). In this paper we consider polynomials $q_n$, $\deg(q_n)\leq n$, satisfying non-Hermitian orthogonality relations with varying weights of the form
\begin{equation}
\label{eq:orthogonality}
\int_Ft^jq_n(t)w_n(t)\frac{dt}{w^+(t)} = 0, \;\;\; j=0,\ldots,n-1,
\end{equation}
together with  their associated functions of the second kind, i.e.
\begin{equation}
\label{eq:secondkind}
R_n(z) := \frac{1}{\pi i}\int_F\frac{q_n(t)w_n(t)}{t-z}\frac{dt}{w^+(t)}, \;\;\; z\in D:=\overline\C\setminus F,
\end{equation}
where $\overline\C$ is the extended complex plane and $\{w_n\}$ is a sequence of complex-valued functions on $F$. Our main goal is to show that the polynomials $q_n$ possess so-called strong (Szeg\H{o}-type) asymptotics under the assumption that $w_n=h_n/v_n$, where $\{h_n\}$ is a compact family of sufficiently smooth functions on $F$ and $v_n$ are polynomials of respective degrees at most $2n$, subject to certain restrictions. The desired result can be expressed as
\begin{equation}
\label{eq:strongasymptotics}
(2/\map)^n(q_n\szf_{w_n}) = 1+o(1), \quad \map:=z+w,
\end{equation}
locally uniformly in $D$, where $\szf_{w_n}$ is the Szeg\H{o} function of $w_n$ (see Section \ref{subsec:strong1}). As a consequence of (\ref{eq:strongasymptotics}), we establish the uniform convergence of multipoint diagonal Pad\'e approximants to Cauchy transforms of complex-valued measures of the form $hdt/w^+$, where $h$ is a sufficiently smooth function on $F$ and the Pad\'e approximants interpolate the corresponding Cauchy transform at the zeros of the polynomials $v_n$.

\subsection{Symmetry Property}
\label{subsec:symmetric}

Let $F$ be as above and $D$ be its complement in $\overline\C$. In what follows, we always assume that the endpoints of $F$ are $\pm1$. Define
\begin{equation}
\label{eq:map}
\map(z) := z+w(z), \quad z\in D,
\end{equation}
where $w$ is given by (\ref{eq:chebyshev}). Then $\map$ is a non-vanishing univalent holomorphic function in $D$ except for a simple pole at infinity. It can be easily checked that
\begin{equation}
\label{eq:boundary}
w^+ = - w^- \;\; \mbox{and therefore} \;\; \map^+\map^-=1 \;\; \mbox{on} \;\; F.
\end{equation}

Let $e\in D$. Define
\[
r(e;z) := \frac{\map(z)-\map(e)}{1-\map(e)\map(z)}, \quad |e|<\infty, \quad \mbox{and} \quad r(\infty;z) := \frac{1}{\map(z)}, \quad z\in D.
\]
Clearly, $r(e;\cdot)$ is a holomorphic function in $D$ with a simple zero at $e$ and non-vanishing otherwise. Moreover, it follows from (\ref{eq:boundary}) that unrestricted boundary values $r^\pm(e;\cdot)$ exist everywhere on $F$ and satisfy
\begin{equation}
\label{eq:boundaryR}
r^+(e;t)r^-(e;t) = 1, \quad t\in F.
\end{equation}

Let now $\E:=\{E_n\}_{n\in\N}$ be a {\it triangular scheme of points} in $D$, i.e. each $E_n$ consists of $2n$ not necessarily distinct nor finite points contained in $D$. We define {\it the support of} $\E$ as $\supp(\E):=\cap_{n\in\N}\overline{\cup_{k\geq n}E_k}\subset\overline\C$. Clearly, the support of any weak$^*$ limit point of the counting measures of points in $E_n$ is a subset of $\supp(\E)$. Hereafter, the counting measure of a finite set is a probability measure that has equal mass at each point counting multiplicities and the weak$^*$ topology is understood with respect to the duality between complex measures and continuous functions with compact support in $\overline\C$. To each set $E_n$ we associate a function $r_n$ by the rule
\begin{equation}
\label{eq:rn}
r_n(z) := \prod_{e\in E_n} r(e;z), \quad z\in D.
\end{equation}
Then $r_n$ is a holomorphic function in $D$ with $2n$ zeros there and whose boundary values on $F$ satisfy $r_n^+r_n^-=1$. Among all possible schemes $\E$, we are interested only in those that induce the following symmetry property on $F$.

\begin{df}[Symmetry w.r.t. $\E$]
Let $F$ be a rectifiable Jordan arc and $\E$ be a triangular scheme of points in $D$. We say that $F$ is symmetric with respect to $\E$ if the associated to $\E$ functions $r_n$ satisfy $|r_n^\pm| = O(1)$ uniformly on $F$ and $r_n=o(1)$ locally uniformly in~$D$.
\end{df}

Note that the boundedness of $|r_n^\pm|$ from above implies their the boundedness away from zero by (\ref{eq:boundaryR}).

As we shall see, this definition is, in fact, closely related to the classical definition of the symmetry of an arc in a field (more generally, of a system of arcs). The latter is based on a number of potential-theoretic notions all of which can be found in \cite{SaffTotik}.

\begin{df}[Symmetry in a Field]
Let $F$ be a rectifiable Jordan arc and let $q$ be a harmonic function in some neighborhood of $F$. It is said that $F$ is symmetric in the field $q$ if the following partial derivatives exist\footnote{Normal derivatives are understood in the strong sense, namely, if the tangent to $F$ exists at $t$ and $\vec{n}_t^\pm$ are the unit normals from each side of $F$ then the limits of the Euclidean scalar products $\langle\nabla(U^\lambda+q)(y),\vec{n}_t^\pm\rangle$ exist as $y$ approaches $t$ along $\vec{n}_t^\pm$, respectively.} and are equal:
\begin{equation}
\label{eq:SProperty}
\frac{\partial(U^\lambda + q)}{\partial n^+} = \frac{\partial(U^\lambda + q)}{\partial n^-} \quad \mbox{a.e. on} \quad \supp(\lambda),
\end{equation}
where $\lambda$ is the weighted equilibrium distribution in the field $q$, $U^\lambda$ is the logarithmic potential of $\lambda$, and $n^\pm$ are the one-sided unit normals to $F$.
\end{df}

The symmetry property turned out to be vital in the investigation of the rates of best approximation of functions with branch points by rational interpolants (multipoint Pad\'e approximants) \cite{St86, St97}. Given a continuum $E$ with connected complement and a function $f$ holomorphic in a neighborhood of $E$, analytically continuable except over a compact set $F_0$, $\cp(F_0)=0$, where $\cp(\cdot)$ is the \emph{logarithmic capacity}, then there exists an \emph{extremal (maximal) domain} $D$ such that the condenser $(F,E)$, $F:=\partial D\supset F_0$, has minimal \emph{condenser capacity} among all domains in which $f$ has single-valued holomorphic continuation. Moreover, if the continuation of $f$ over $\overline\C\setminus F_0$ is not single-valued then the boundary $F$ of the extremal domain is a union of a system of open analytic arcs and a set of capacity zero, and it can be characterized as the unique contour containing $F_0$ and satisfying (\ref{eq:SProperty}) with $q=-U^{\nu_E}$, where $\nu_E-\nu_F$ is the equilibrium distribution for the minimal energy problem for signed measures of the form $\sigma_E-\sigma_F$, where $\sigma_E$ and $\sigma_F$ range over all probability Borel measures on $E$ and $F$, respectively. The best rate (in the $n$-th root sense) of approximation to $f$ on $E$ is achieved by rational interpolants corresponding to a triangular scheme whose points are asymptotically distributed as $\nu_E$. The counting measures of the poles of such interpolants weakly converge to $\nu_F$ and the interpolants themselves converge to $f$ in capacity in $D$.

Dwelling on the work of H. Stahl discussed above, A. A. Gonchar and E. A. Rakhmanov \cite{GRakh87} extended the definition of a symmetric contour to general harmonic fields in order to give a complete formal proof of the ``Magnus conjecture'' on the rate of approximation of the exponential function on the positive semi-axis\footnote{This conjecture was formerly known as ``1/9 conjecture'', which was shown to be false by H.-U. Opitz and K. Scherer in \cite{OpSch85}. The correct constant was identified by A. P. Magnus in \cite{Mag86}, however, his proof was not entirely formal.}. In \cite{GRakh87} it is assumed that $F$ and $q$ are such that $\supp(\lambda)$ is a \emph{tame} set, i.e. the intersection of a sufficiently small neighborhood of quasi every point (q.e. means up to a set of zero capacity) of $\supp(\lambda)$ with the support itself is an analytic arc. It is easy to check that in this case partial derivatives in (\ref{eq:SProperty}) do exist. It was also observed that in the setting of rational interpolation to functions as above, one can take interpolation points asymptotically distributed like any Borel measure $\nu$ on $E$. This will define a unique contour $F\supset F_0$ that satisfies (\ref{eq:SProperty}) with $q=-U^\nu$ and the interpolants will converge to the approximated function in capacity in $\overline\C\setminus F$ while the counting measures of their poles will weakly converge to $\hat\nu$, where $\hat\nu$ is the \emph{balayage} of $\nu$ onto $F$.

Since our main interest lies with rational interpolation, we shall be concerned only with arcs satisfying (\ref{eq:SProperty}), where $q=-U^\nu$ for some Borel measure $\nu$ with compact support. In this case much milder assumptions on $F$ are sufficient in order to have well-defined partial derivatives (see the statement of the next theorem). Necessarily, for such fields, the symmetric arcs turn out to be analytic as apparent from the following theorem.

\begin{thm}
\label{thm:sp}
Let $F$ be a rectifiable Jordan arc such that for $x=\pm1$ and all $t\in F$ sufficiently close to $x$ it holds that $|F_{t,x}|\leq\const|x-t|^\beta$, $\beta>1/2$, where $|F_{t,x}|$ is the arclenth of the subsarc of $F$ joining $t$ and $x$ and $\const$ is an absolute constant. Then the following are equivalent:
\begin{itemize}
\item[(a)] there exists a triangular scheme of points $\E$, $\supp(\E)\subset D$, such that $F$ is symmetric with respect to $\E$;
\item[(b)] there exists a positive compactly supported Borel measure $\nu$, $\supp(\nu)\subset D$, such that $F$ is symmetric in the field $-U^\nu$;
\item[(c)] $F$ is an analytic Jordan arc, i.e. there exists a univalent function $p$ holomorphic in some neighborhood of $[-1,1]$ such that $F=p([-1,1])$;
\end{itemize}
\end{thm}

The above theorem covers only the case where $\supp(\E)$ is disjoint from $F$. The authors do not know whether non-analytic arcs can be symmetric with respect to a triangular scheme when the support of the latter does intersect the arc.

We point out that the proof of Theorem \ref{thm:sp} is constructive. In other words, for a given analytic arc $F$, a suitable scheme $\E$ can (in a non-unique manner) be explicitly written in terms of the function $p$ that analytically parametrizes $F$. Each such scheme $\E$ gives rise to a suitable measure $\nu$ simply by taking the weak$^*$ limit of the counting measures of points of the sets in $\E$.

\subsection{Strong Asymptotics for Non-Vanishing Densities}
\label{subsec:strong1}

Let $K$ be a compact set in $\C$ and denote by $C(K)$ the space of continuous functions on $K$ endowed with the usual supremum norm $\|\cdot\|_K$. For $h\in C(K)$, set $\omega_h$ to be the modulus of continuity of $h$, i.e.
\[
\omega_h(\tau) := \max_{|t_1-t_2|\leq \tau}|h(t_1)-h(t_2)|, \;\;\; \tau\in[0,\diam(K)],
\]
where $\diam(K) := \max_{t_1,t_2\in K}|t_1-t_2|$. It is said that $h$ is {\it Dini-continuous} on $K$ if
\[
\int_{[0,\diam(K)]}\frac{\omega_h(\tau)}{\tau}d\tau<\infty.
\]
We denote by $\dc(K)\subset C(K)$ the set of Dini-continuous functions on $K$ and by $\dc_\omega(K)$ the subset of $\dc(K)$ such that $\omega_h\leq\omega$ for every $h\in\dc(K)$, where $\omega$ is the modulus of continuity of \emph{some} function in $\dc(K)$. In what follows, we shall employ the symbol ``$^*$'' to indicate the nowhere vanishing subset of a functional set (for instance, $C^*(K)$ stands for the non-vanishing continuous functions on $K$).

Let $h\in\dc^*(F)$ and $\log h$ be an arbitrary but fixed continuous branch of the logarithm of $h$. Then it is easily verified (see Section \ref{subsec:aux_szf}) that the {\it geometric mean} of $h$, i.e.
\begin{equation}
\label{eq:geommean}
\gm_h := \exp\left\{\int_F\log h(t)\frac{idt}{\pi w^+(t)}\right\},
\end{equation}
is independent of the actual choice of the branch of the logarithm and is non-zero. Moreover, the {\it Szeg\H{o} function} of $h$, i.e.
\begin{equation}
\label{eq:szego}
\szf_h(z) := \exp\left\{\frac{w(z)}{2}\int_F\frac{\log h(t)}{z-t}\frac{idt}{\pi w^+(t)}-\frac12\int_F\log h(t)\frac{idt}{\pi w^+(t)}\right\}, \quad z\in D,
\end{equation}
is also independent of the choice of the branch (as long as the same branch is used in both integrals). In fact, it is the unique non-vanishing holomorphic function in $D$ that has continuous unrestricted boundary values on $F$ from each side and satisfies
\begin{equation}
\label{eq:szegodecomp}
h = \gm_h\szf_h^+\szf_h^- \;\; \mbox{on} \;\; F \;\; \mbox{and} \;\; S_h(\infty)=1.
\end{equation}

To state our next theorem we need one more notion. Let $X$ be a Banach space and $Y$ be a subset of $X$. We say that a family $\{h_n\}\subset Y$ is {\it relatively compact in} $Y$ if any sequence from this family contains a norm-convergent subsequence and all the limit points of $\{h_n\}$ belong to $Y$.

With the use of the previous notation, we formulate our first result on strong asymptotics of non-Hermitian orthogonal polynomials.

\begin{thm}
\label{thm:sa1}
Let $F$ be an analytic Jordan arc connecting $\pm1$ that is symmetric with respect to a triangular scheme of points $\E=\{E_n\}_{n\in\N}$. Further, let $\{q_n\}_{n\in\N}$ be a sequence of polynomials satisfying
\[
\int_Ft^jq_n(t)w_n(t)\frac{dt}{w^+(t)} = 0, \;\;\; j=0,\ldots,n-1,
\]
with $w_n=h_n/v_n$, where $\{h_n\}$ is a relatively compact family in $\dc^*_\omega(F)$ and $v_n$ are monic polynomials with zeros at the finite points of $E_n$. Then, for all $n$ large enough, polynomials $q_n$ have exact degree $n$ and therefore can be normalized to be monic. Under such a normalization, we have that
\begin{equation}
\label{eq:sa1}
\left\{
\begin{array}{lll}
q_n  &=& [1+o(1)]/\szf_n \\
R_nw &=& [1+o(1)]\gamma_n\szf_n
\end{array}
\right. \;\;\; \mbox{locally uniformly in} \;\; D,
\end{equation}
where $\szf_n := (2/\map)^n\szf_{w_n}$, $\gamma_n:=2^{1-2n}\gm_{w_n}$, and $R_n$ was defined in \eqref{eq:secondkind};
\begin{equation}
\label{eq:sa2}
\left\{
\begin{array}{lll}
q_n &=&  (1+d_n^-)/\szf_n^+ + (1+d_n^+)/\szf_n^- \\
(R_nw)^\pm &=& (1+d_n^\pm)~\gamma_n\szf_n^\pm
\end{array}
\right. \;\;\; \mbox{on} \;\; F,
\end{equation}
where $d_n^\pm\in\cs$ and satisfy\footnote{From classical estimates on singular integrals \cite{Gakhov} the functions $(R_nw)^\pm$ are, in fact, continuous functions on $F$; this also can be seen, for instance, from equation \eqref{eq:translation} below. However, \eqref{eq:sa2} does not contain explicit information on the pointwise behavior of $(R_nw)^\pm$ as the smallness of $d_n$ is claimed in $L^p$ norm only and pays no particular attention to the values of $d_n$ on a set of zero linear measure of $F$.}
\[
\int_F\frac{|d_n^-(t)|^p + |d_n^+(t)|^p}{\sqrt{|1-t^2|}}|dt| \to 0 \;\; \mbox{as} \;\; n\to\infty
\]
for any $p\in[1,\infty)$. Furthermore, the following limit takes place
\begin{equation}
\label{eq:sa3}
\frac{q_n^2(t)w_n(t)}{\gamma_nw^+(t)}dt ~\cws~ \frac{dt}{w^+(t)},
\end{equation}
where $dt$ is the differential along $F$ and ``$\cws$'' stands for the weak$^*$ convergence of measures.
\end{thm}

The theorem is stated for analytic arcs only because in the view of Theorem \nolinebreak \ref{thm:sp} this will be the case if the scheme $\E$ does not touch $F$, which is a standard setting in interpolation. In fact, a careful examination of the proof of Theorem \ref{thm:sa1} shows that $F$ could simply be a Carleson (Ahlfors-regular) curve. However, the authors do not have any convincing evidence suggesting that such arcs could be symmetric with respect to some triangular scheme unless they are analytic.

\subsection{Strong Asymptotics for Densities with Some Vanishing}
\label{subsec:strong2}

In this section we consider an extension of Theorem \ref{thm:sa1} to densities that may vanish. As mentioned in the introduction, such an extention is important since it provides the first example of strong asymptotics with vanishing density. The method used in the proof is essentially the same and is presented separately solely for the clarity of the exposition.

The purpose of the following construction is to introduce a factor having zeros in the density of the measure and whose vanishing is sufficiently weak so that the proof of Theorem \ref{thm:sa1} still goes through with minor modifications.

Denote by $F_{c,d}$ the closed subarc of $F$ with endpoints $c$ and $d$. As $F$ is compact there exists $y\in(-\infty,-1]$ such that $F\cap(-\infty,y)=\emptyset$. Chose $F_{aux}$ to be any smooth Jordan arc, $F_{aux}\cap F=\{-1\}$, that links $y$ and $-1$. Let $x\in F$, define $\ar_x$ to be the branch of $\arg(\cdot-x)$ that is continuous in $\C\setminus\{(-\infty,y]\cup F_{aux} \cup F_{-1,x}\}$ and tends to zero when the variable approaches infinity along the positive real axis. For definiteness, we set $\ar_x(x)$ to be the limit of $\ar_x$ along $F_{x,1}$. Then $\ar_x$ has continuous traces $\ar_x^\pm$ on $F_{-1,x}$ oriented from $-1$ to $x$ and $\ar_x^+(t) = \ar_x^-(t) + 2\pi$, $t\in F_{-1,x}\setminus\{x\}$. Clearly, the functions $\ar_x^\pm$ and $\ar_{x|F_{x,1}}$ do not depend on the choice of $F_{aux}$.

For any $\alpha\in(0,1/2]$ and $x\in F$ set
\begin{equation}
\label{eq:hbar}
\hbar(\alpha,x;t) := |2(t-x)|^{2\alpha}\left\{
\begin{array}{ll}
\exp\left\{2i\alpha\ar_x(t)\right\}, & t\in F_{x,1}, \\
\exp\left\{2i\alpha(\ar^+_x(t)-\pi)\right\}, & t\in F_{-1,x}\setminus\{x\}.
\end{array}
\right.
\end{equation}
Then $\hbar(1,x;t) = 4(t-x)^2$, $t\in F$, the argument of $\hbar(\alpha,x;\cdot)$ is continuous on $F$, and $\hbar(\alpha,x;t) = |2(t-x)|^{2\alpha}$ when $F=[-1,1]$. 

Let $h$ be a function on $F$ for which $\szf_h$ is well-defined. Set $\scf_h^\pm := \szf_h^\pm/\szf_h^\mp$. It is shown in Section \ref{subsec:aux_szf} that functions $\scf_h^\pm$ are continuous on $F$ and satisfy $\scf_h^\pm(\pm1)=1$ whenever $h\in\dc^*(F)$. Moreover, it is also shown that the functions $\scf_{\hbar(\alpha,x;\cdot)}^\pm$ are continuous on $F\setminus\{x\}$ and have moduli $|\map^\mp(t)|^{2\alpha}$, $t\in F\setminus\{x\}$, where $\map$ was defined in (\ref{eq:map}). They possess the one-sided limits along $F$ at $x$, denoted by $\scf_{\hbar(\alpha,x;\cdot)}^\pm(x^\pm)$, and
\begin{equation}
\label{eq:scatjumps}
\left|\scf_{\hbar(\alpha,x;\cdot)}^\pm(x^+)-\scf_{\hbar(\alpha,x;\cdot)}^\pm(x^-)\right| = 2\sin(\alpha\pi)\left|\map^\mp(x)\right|^{2\alpha},
\end{equation}
where we make the convention that
\[
\scf_{\hbar(\alpha,x;\cdot)}^+(1^+):=\scf_{\hbar(\alpha,x;\cdot)}^-(1^-) \quad \mbox{and} \quad \scf_{\hbar(\alpha,x;\cdot)}^-(-1^-):=\scf_{\hbar(\alpha,x;\cdot)}^+(-1^+).
\]

Now, let $F_0\subset F$ be a finite set of distinct points. We associate to each $x\in F_0$ some $\alpha_x\in(0,1/2)$ and define
\begin{equation}
\label{eq:VF}
\hbar(t) = \hbar(F_0;t) := \prod_{x\in F_0\setminus\{\pm1\}}\hbar(\alpha_x,x;t)\prod_{x\in F_0\cap\{\pm1\}}\hbar(\alpha_x/2,x;t),
\end{equation}
$t\in F$. Then the following theorem takes place.

\begin{thm}
\label{thm:sa2}
Assume that
\begin{itemize}
\item $F$ is an analytic Jordan arc connecting $\pm1$ that is symmetric with respect to a triangular scheme of points $\E=\{E_n\}_{n\in\N}$ and let $v_n$ be the monic polynomial with zeros at the finite points of $E_n$;
\item the functions $r_n$, associated to $\E$ via \eqref{eq:rn}, are such that $\{|r_n\circ\map^{-1}|\}_{n\in\N}$ is a relatively compact family in $\dc(\Gamma)$, where $\Gamma$ is the boundary of $\map(D)$;
\item $\{h_n\}_{n\in\N}$ is a relatively compact family in $\dc^*_\omega(F)$;
\item $F_0\subset F$ is a finite set of distinct points, $\alpha_x\in(0,1/2)$, $x\in F_0$, and $\hbar$, given by \eqref{eq:VF}, is such that
\begin{equation}
\label{eq:upsilon}
\limsup_{n\to\infty} \left|(\scf_{\hbar h_n}r_n)^\pm(x^+)-(\scf_{\hbar h_n}r_n)^\pm(x^-)\right| < 2\|\oq\|^{-1},
\end{equation}
$x\in F_0$, where $1\leq\|\oq\|$ is the norm of the outer Cauchy projection operator on $L^2(\Gamma)$ defined in \eqref{eq:opoq}.
\end{itemize}
If $\{q_n\}_{n\in\N}$ is a sequence of polynomials satisfying \eqref{eq:orthogonality} with $w_n=\hbar h_n/v_n$, then the polynomials $q_n$ have exact degree $n$ for all $n$ large enough and therefore can be normalized to be monic. Under such a normalization \eqref{eq:sa1} and \eqref{eq:sa2} hold with $d_n$ satisfying
\[
\limsup_{n\to\infty} \int_F\frac{|d_n^-(t)|^2+|d_n^+(t)|^2}{\sqrt{|1-t^2|}}|dt| < \infty.
\]
\end{thm}

Let us make several remarks. As in the case of Theorem \ref{thm:sa1}, the requirement of analyticity of $F$ can be weakened. It is enough to assume $F$ to be piecewise Dini-smooth without cusps. This condition is sufficient for Section \ref{subsec:aux_cf} to remain valid, which is the main additional ingredient in the proof of Theorem \ref{thm:sa2} as compared to the proof of Theorem \ref{thm:sa1}.

As clear from the definition of $\oq$, the lower bound in the inequality $\|\oq\|\geq1$ is achieved when $F=[-1,1]$.

To obtain an explicit upper bound for the numbers $\alpha_x$, $x \in F_0$, for which the theorem holds, rewrite (\ref{eq:upsilon}) in view of (\ref{eq:scatjumps}) as
\[
\sin(\alpha_x\pi)\left|\map^\mp(x)\right|^{2\sum_{y\in F_0}\alpha_y} \limsup_{n\to\infty}|(\scf_{h_n}r_n)^\pm(x)| < \|\oq\|^{-1}.
\]
Thus, since $2\alpha_y<1$, $y\in F_0$, (\ref{eq:upsilon}) is satisfied for example when
\[
\sin(\alpha_x\pi) < \|\oq\|^{-1} \min\{|\map^-(x)|^L,|\map^+(x)|^L\} \liminf_{n\to\infty}|(\scf_{h_n}r_n)^\pm(x)|
\]
where $L$ is the number of elements in $F_0$. In particular, if $F=[-1,1]$, the functions $h_n$ are positive, and the sets $E_n$ are conjugate-symmetric, then $|\map^\pm|\equiv|\scf_{h_n}^\pm|\equiv|r_n^\pm|\equiv1$ and (\ref{eq:upsilon}) simply reduces to the initial condition $2\alpha_x<1$, $x\in F_0$.

\subsection{Multipoint Pad\'e Approximation}
\label{subsec:pade}

Let $\mu$ be a complex Borel measure with compact support. We define the Cauchy transform of $\mu$ as
\begin{equation}
\label{eq:CauchyT}
f_\mu(z) := \int\frac{d\mu(t)}{z-t}, \quad z\in\overline\C\setminus\supp(\mu).
\end{equation}
Clearly, $f_\mu$ is a holomorphic function in $\overline\C\setminus\supp(\mu)$ that vanishes at infinity.

Classically, diagonal (multipoint) Pad\'e approximants to $f_\mu$ are rational functions of type $(n,n)$ that interpolate $f_\mu$ at a prescribed system of $2n+1$ points. However, when the approximated function is of the from (\ref{eq:CauchyT}), it is customary to place at least one interpolation point at infinity. More precisely, let $\E=\{E_n\}$ be a triangular scheme of points in $\overline\C\setminus\supp(\mu)$ and let $v_n$ be the monic polynomial with zeros at the finite points of $E_n$.

\begin{df}[Multipoint Pad\'e Approximant]
Given $f_\mu$ of type \eqref{eq:CauchyT} and a triangular scheme $\E$, the $n$-th diagonal Pad\'e approximant to $f_\mu$ associated with $\E$ is the unique rational function $\Pi_n=p_n/q_n$ satisfying:
\begin{itemize}
\item $\deg p_n\leq n$, $\deg q_n\leq n$, and $q_n\not\equiv0$;
\item $\left(q_n(z)f_\mu(z)-p_n(z)\right)/v_n(z)$ is analytic in $\overline\C\setminus\supp(\mu)$;
\item $\left(q_n(z)f_\mu(z)-p_n(z)\right)/v_n(z)=O\left(1/z^{n+1}\right)$ as $z\to\infty$.
\end{itemize}
\end{df}

A multipoint Pad\'e approximant always exists since the conditions for $p_n$ and $q_n$ amount to solving a system of $2n+1$ homogeneous linear equations with $2n+2$ unknown coefficients, no solution of which can be such that $q_n\equiv0$ (we may thus assume that $q_n$ is monic); note that the required interpolation at infinity is entailed by the last condition and therefore $\Pi_n$ is, in fact, of type $(n-1,n)$.

The following theorem is an easy consequence of Theorems \ref{thm:sp} and \ref{thm:sa1}.

\begin{thm}
\label{thm:pade}
Let $F$ be an analytic Jordan arc connecting $\pm1$. Then there always exist triangular schemes such that $F$ is symmetric with respect to them. Let $\E$ be any such scheme and $f_\mu$ be given by \eqref{eq:CauchyT} with
\begin{equation}
\label{eq:measuremu}
d\mu(t) = \dot\mu(t)\frac{idt}{\pi w^+(t)}, \quad \dot\mu\in\dc^*(F), \quad \supp(\mu) = F.
\end{equation}
Then the sequence of diagonal Pad\'e approximants to $f_\mu$ associated with $\E$, $\{\Pi_n\}$, is such that
\begin{equation}
\label{eq:speedconv}
(f_\mu-\Pi_n)w = [2\gm_{\dot\mu} + o(1)]\szf^2_{\dot\mu}r_n \quad \mbox{locally uniformly in} \quad D.
\end{equation}
\end{thm}

The proof of Theorem \ref{thm:pade} combined with Theorem \ref{thm:sa2} yields the following.

\begin{cor}
Let $F$ an analytic arc connecting $\pm1$ that is symmetric with respect to a triangular scheme $\E$. Further, let $h\in\dc^*(F)$, $F_0\subset F$ be a finite set of distinct points, and $\hbar(F_0;\cdot)$ be given by \eqref{eq:VF} for some choice of the parameters $\alpha_x<1/2$, $x\in F_0$. If functions $r_n$, associated to $\E$, are such that $\{|r_n\circ\map^{-1}|\}$ is relatively compact in $\dc(\Gamma)$, $\Gamma=\partial \map(D)$, and \eqref{eq:upsilon} is satisfied with $h_n=h$ and $\hbar=\hbar(F_0;\cdot)$, then the conclusion of Theorem \ref{thm:pade} remains valid for $f_\mu$ with $\dot\mu=\hbar h$.
\end{cor}

As a final remark of this section, let us mention that it is possible to consider in Theorem \ref{thm:pade} the Radon-Nikodym derivative of $\mu$ with respect to the equilibrium distribution on $F$. The later will be different from $\dot\mu$ by a factor $g(\map)\map$, where $g$ is a non-vanishing function analytic across $\Gamma$. However, as apparent from (\ref{eq:speedconv}), the adopted decomposition of $\mu$ is more convenient.

\section{Auxiliary Material}
\label{sec:aux}

In this section we provide the necessary material to proceed with the proofs of the main results. Section \ref{subsec:aux_jouk} describes properties of $\Gamma$, the boundary of $\map(D)$. In Sections \ref{subsec:aux_cio} and \ref{subsec:aux_szf} we constructively define Szeg\H{o} functions, show that formula (\ref{eq:szegodecomp}) indeed takes place, and compute several examples that we use throughout the proofs. The main source for these sections is the excellent monograph \cite{Gakhov}. Sections \ref{subsec:aux_smirnov} and \ref{subsec:aux_th} are concerned with the so-called Smirnov classes on Jordan curves and their content can be found in much greater detail in \cite{BottcherKarlovich}. Finally, Section \ref{subsec:aux_cf} carries onto analytic Jordan domains certain properties of inner-outer factorization of analytic functions on the unit disk.

\subsection{Joukovski Transformation}
\label{subsec:aux_jouk}

Let $F$ be a rectifiable Jordan arc connecting $\pm1$, oriented from $-1$ towards $1$, and let $\map$ be defined by (\ref{eq:map}). Then $\map$ has continuous injective extensions $\map^\pm$ onto $F$ that are conformal on $F\setminus\{\pm1\}$. Denote
\[
\Gamma: = \Gamma^+\cup\Gamma^-, \;\;\; \Gamma^\pm:=\map^\pm(F).
\]
By what precedes, $\Gamma^\pm\setminus\{\pm1\}$ are open Jordan arcs with endpoints $\pm1$ and by (\ref{eq:boundary})
\begin{equation}
\label{eq:reciprocity}
\Gamma^-=(\Gamma^+)^{-1} := \left\{z:~ 1/z\in\Gamma^+\right\}.
\end{equation}
Observe that $\map$ is inverse to $\jt(z):=(z+1/z)/2$, the usual Joukovski transformation, in $D=\overline\C\setminus F$. Since $\jt$ maps $z$ and $1/z$ into the same point and $\map^\pm(z)=\pm1$ if and only if $z=\pm1$, $\Gamma^+ \cap \Gamma^- = \{-1,1\}$ and therefore $\Gamma$ is a Jordan curve.

Assume, in addition, that for $x=\pm1$ and all $t\in F$ sufficiently close to $x$ it holds that $|F_{t,x}|\leq\const|x-t|^\beta$, $\beta>1/2$, where $|F_{t,x}|$ is the arclenth of the subsarc of $F$ joining $t$ and $x$ and $\const$ is some absolute constant. We shall show that this implies rectifiability of $\Gamma$. As $\map^\pm$ are conformal on $F\setminus\{\pm1\}$, it is enough to prove that $\Gamma$ has finite length around $\pm1$. By (\ref{eq:reciprocity}), it is, in fact, sufficient to consider only $\Gamma^+$. Let $s$ be a parametrization of $F$ around $1$ such that $s(0)=1$ and for all $\delta\in[0,\delta_0]$, $s(\delta_0)\neq-1$, it holds that
\[
\ell(\delta) := |F_{s(\delta),1}| = \int_{[0,\delta]}|s^\prime(y)|dy \leq \const|1-s(\delta)|^\beta.
\]
As $\map^+\circ s$ is a parametrization of $\Gamma^+$ around 1, we get that
\begin{eqnarray}
|\map^+(F_{s(\delta),1})| &=& \int_{[0,\delta]}|\left(\map^+(s(y))\right)^\prime|dy = \int_{[0,\delta]}\left|\frac{\map^+(s(y))}{w^+(s(y))}\right||s^\prime(y)|dy \nonumber \\
{} &\leq& \const \int_{[0,\delta]}\frac{|s^\prime(y)|}{\sqrt{|1-s(y)|}}dy \nonumber
\end{eqnarray}
since $(\map^+)^\prime=\map^+/w^+$, $|\map^+|$ is bounded above on $F$, and $|1+s(y)|$ is bounded away from zero on $[0,\delta_0]$. As $\ell^\prime(y)=|s^\prime(y)|$  and $|1-s(y)|\geq \const\ell^{1/\beta}(y)$ on $[0,\delta_0]$, we deduce that
\[
|\map^+(F_{s(\delta),1})| \leq \const \int_{[0,\delta]}\frac{\ell^\prime(y)}{\ell^{1/2\beta}(y)}dy = \const\ell^{1-1/2\beta}(\delta) < \infty
\]
since $1-1/2\beta>0$. Clearly, analogous bound holds around $-1$. Hence, $\Gamma^+$ and therefore $\Gamma$ are rectifiable.

Now, we show that $\Gamma$ is an analytic Jordan curve whenever $F$ is an analytic Jordan arc. In other words, we show that in this case there exists a holomorphic univalent function in some neighborhood of the unit circle that maps $\T$ onto $\Gamma$. Let $U$ be a neighborhood of $\T$ such that $\jt(U)$ lies in the domain of $p$, where $p$ is a holomorphic univalent parametrization of $F$, $F=p([-1,1])$. Define
\[
\Phi(z) := \left\{
\begin{array}{ll}
1/\map(p(\jt(z))),   & z\in U^+:=U\cap \D, \\
\map(p(\jt(z))), & z\in U^-:=U\setminus\overline\D.
\end{array}
\right.
\]
Then $\Phi$ is a sectionally holomorphic function on $U\setminus\T$. Denote by $\Phi^\pm$ the traces of $\Phi$ from $U^\pm$ on $\T$. Then for $\tau\in\T$, we have
\begin{eqnarray}
\Phi^-(\tau) &=& \left\{
\begin{array}{ll}
1/\map^+(p(\jt(\tau))), & \im(\tau)\geq0 \\
1/\map^-(p(\jt(\tau))), & \im(\tau)<0
\end{array}
\right. \nonumber \\
{} &=& \left\{
\begin{array}{ll}
\map^-(p(\jt(\tau))), & \im(\tau)\geq0 \\
\map^+(p(\jt(\tau))), & \im(\tau)<0
\end{array}
\right. = \Phi^+(\tau). \nonumber
\end{eqnarray}
Thus, $\Phi$ is a holomorphic injective function on $U$ that maps $\T$ into $\Gamma$ and therefore $\Gamma$ is an analytic Jordan curve.

\subsection{The Cauchy Integral Operator}
\label{subsec:aux_cio}

Let $\Gamma$ be an analytic Jordan curve and denote by $D^+$ and $D^-$ the bounded and unbounded components of the complement of $\Gamma$, respectively. Further, let $\phi$ be an integrable function on $\Gamma$. The {\it Cauchy integral operator} on $\Gamma$ is defined as
\begin{equation}
\label{eq:oc}
\oc\phi(z) := \frac{1}{2\pi i}\int_\Gamma\frac{\phi(\tau)}{\tau-z}d\tau, \;\;\; z\not\in\Gamma.
\end{equation}
Clearly, $\oc\phi$ is a sectionally holomorphic function on $\C\setminus\Gamma$ and
\[
\oc\phi(z) = \left\{
\begin{array}{ll}
\phi(t) + A_\phi(t,z), & z\in D^+, \\
A_\phi(t,z), & z\in D^-,
\end{array}
\right. \quad t\in \Gamma,
\]
where
\[
A_\phi(t,z) := \frac{1}{2\pi i}\int_\Gamma \frac{\phi(\tau)-\phi(t)}{\tau-z}d\tau.
\]
It is an easy modification of \cite[Sec. 4.1]{Gakhov} to see that $A_\phi(t,z)$ is a continuous function of $z$ for each $t\in\Gamma$ when $\phi\in\dc(\Gamma)$. Therefore, in this case, $\oc\phi_{|D^\pm}$ have continuous unrestricted boundary values $(\oc\phi)^\pm$ such that
\begin{equation}
\label{eq:pls}
(\oc\phi)^+ - (\oc\phi)^- = \phi.
\end{equation}
Moreover, $\oc\phi$ is the unique sectionally holomorphic function in $\overline\C\setminus\Gamma$ that satisfies (\ref{eq:pls}) and vanishes at $\infty$.

Let now $\{\phi_n\}\subset\dc_\omega(\Gamma)$ be a uniformly convergent sequence. Denote by $\phi$ the limit function of $\{\phi_n\}$ and observe that $\phi$ necessarily belongs to $\dc_\omega(\Gamma)$. We want to show that $(\oc\phi_n)^\pm$ converge uniformly on $\Gamma$ to $(\oc\phi)^\pm$, respectively. Since $\oc$ is additive, we may suppose without loss of generality that $\phi\equiv0$. Then
\begin{eqnarray}
|(\oc\phi_n)^-(t)| &=& |A_{\phi_n}(t,t)| \leq \int_\Gamma \frac{|\phi_n(\tau)-\phi_n(t)|}{|\tau-t|}\frac{|d\tau|}{2\pi} \nonumber \\
{} &\leq& k_1\int_{[0,s_0]}\frac{\omega_{\phi_n}(s)}{s}ds = k_1\left(\int_{[0,s_n]}+\int_{[s_n,s_0]}\right)\frac{\omega_{\phi_n}(s)}{s}ds \nonumber \\
{} &\leq& k_1\int_{[0,s_n]}\frac{\omega(s)}{s}ds + 2k_1(s_0-s_n)\frac{\|\phi_n\|_\Gamma}{s_n}, \nonumber
\end{eqnarray}
where $s_0:=\diam(\Gamma)$ and $k_1$ is a positive constant that depends only on $\Gamma$. Now, by choosing $s_n$ in such a manner that $s_n$ and $\|\phi_n\|_\Gamma/s_n$ both converge to zero as $n\to\infty$, we obtain that $(\oc\phi_n)^-$ converges to zero uniformly on $\Gamma$. The convergence of $(\oc\phi_n)^+$ then follows from (\ref{eq:pls}).

Finally, suppose that $\Gamma$ is such that $\tau$ and $1/\tau$ belong to $\Gamma$ simultaneously. Assume in addition that $\phi\in\dc(\Gamma)$ is such that $\phi(\tau)=\phi(1/\tau)$. Then
\[
\oc\phi(1/z) = \oc\phi(0) - \oc\phi(z), \quad z\notin\Gamma,
\]
$(\oc\phi)^\pm\in C(\Gamma)$, and by (\ref{eq:pls})
\begin{equation}
\label{eq:sym_pls}
(\oc\phi)^+(\tau) + (\oc\phi)^+(1/\tau) = \phi(\tau)+\oc\phi(0), \quad \tau\in\Gamma.
\end{equation}
Suppose further that $\pm1\in\Gamma$ and denote $F=\jt(\Gamma)$. Then $F$ is an analytic Jordan arc with endpoints $\pm1$. Let $\psi\in\dc(F)$ and $\phi = \psi \circ \jt$. Clearly, $\phi\in\dc(\Gamma)$ and $\phi(\tau)=\phi(1/\tau)$, $\tau\in\Gamma$. Set
\[
R(\zeta) := \frac{w(\zeta)}{2\pi i}\int_F \frac{\psi(t)}{t-\zeta}\frac{dt}{w^+(t)}, \quad \zeta\in D=\overline\C\setminus F.
\]
As
\[
\frac{d\jt(\tau)}{w^+(\jt(\tau))} = \left\{
\begin{array}{ll}
-d\tau/\tau, & \tau\in\Gamma^- \\
d\tau/\tau, & \tau\in\Gamma^+
\end{array}
\right.,
\]
where $\Gamma^\pm$ have the same meaning as in Section \ref{subsec:aux_jouk}, and since $w(\zeta) = (1-z^2)/(2z)$, $\zeta=\jt(z)$, $z\in D^+$, we have for $t=\jt(\tau)$ that
\begin{eqnarray}
R(\zeta) &=& \frac{1-z^2}{4\pi i}\int_\Gamma \frac{\phi(\tau)}{(\tau-z)(1-z\tau)}d\tau = \frac{1}{4\pi i}\int_\Gamma \frac{\phi(\tau)}{\tau-z}d\tau + \frac{z}{4\pi i}\int_\Gamma \frac{\phi(\tau)}{1-z\tau}d\tau \nonumber \\
\label{eq:translation}
{} &=& \frac{1}{2\pi i}\int_\Gamma \frac{\phi(\tau)}{\tau-z}d\tau - \frac{1}{4\pi i}\int_\Gamma \frac{\phi(\tau)}{\tau}d\tau = \oc\phi(z) - \frac12\oc\phi(0),
\end{eqnarray}
where we used the symmetry of $\phi$, i.e. $\phi(\tau)=\phi(1/\tau)$. Then we get from (\ref{eq:translation}) and (\ref{eq:sym_pls}) that
\begin{equation}
\label{eq:wSokhotski}
R^\pm\in C(F), \quad R^+(\pm1)=R^-(\pm1), \quad \mbox{and} \quad R^+ + R^- = \psi \quad \mbox{on} \quad F.
\end{equation}

\subsection{Szeg\H{o} Functions}
\label{subsec:aux_szf}

Let $F$ be an analytic oriented Jordan arc connecting $-1$ and $1$ and $h\in\dc^*(F)$. Since $|\Log(1+z)| \leq 2|z|$, $|z|<1/2$, where $\Log$ is the principal branch of the logarithm, we have that
\begin{equation}
\label{eq:modcontlog}
|\log h(t_1)-\log h(t_2)| = |\Log(h(t_1)/h(t_2))| \leq 2\frac{\omega_h(|t_1-t_2|)}{\min_F h}, 
\end{equation}
whenever $2\omega_h(|t_1-t_2|) \leq \min_F h$. Thus, $\log h\in\dc(F)$ for any continuous branch of the logarithm of $h$. Let now $\Gamma$ be the preimage of $F$ under the Joukovski transformation, $\jt$, and $\Gamma^\pm=\map^\pm(F)$. Further, let $\phi := \log h \circ \jt$. Then, changing variables as in (\ref{eq:translation}), we have that
\[
\gm_h = \exp\left\{\int_F \log h(t) \frac{idt}{\pi w^+(t)}\right\} = \exp\left\{\int_\Gamma \frac{\phi(\tau)}{\tau}\frac{d\tau}{2\pi i}\right\} = \exp\left\{\oc\phi(0)\right\} \neq 0
\]
and does not depend on the branch of the logarithm we chose. Indeed, different continuous branches of the logarithm of $h$ give rise to the same $\phi$ up to an additive integer multiple of $2\pi i$, and since $\oc(\phi+c)(0)=\oc\phi(0)+c$ for any constant $c$, the claim follows. Further, (\ref{eq:translation}) yields that
\[
\szf_h(\zeta) = \exp\left\{\oc\phi(z)-\oc\phi(0)\right\}, \quad \zeta=\jt(z)\in D, \quad z\in D^+.
\]
Hence, $\szf_h$ is a well-defined non-vanishing holomorphic function in $D$ that has value 1 at infinity and satisfies
\begin{equation}
\label{eq:boundaryvalues}
\szf_h^+(t)\szf_h^-(t) =  \exp\left\{\phi(\tau)-\oc\phi(0)\right\} = h(t)/\gm_h, \quad t=\jt(\tau)\in F,
\end{equation}
by (\ref{eq:translation}) and (\ref{eq:sym_pls}). Moreover, (\ref{eq:wSokhotski}) tells us that
\begin{equation}
\label{eq:endpoints}
\szf_h^\pm\in C(F) \quad \mbox{and} \quad \szf_h^+(\pm1) = \szf_h^-(\pm1).
\end{equation}
Now, if $h_1,h_2\in\dc^*(F)$ then $\szf_{h_1h_2} = \szf_{h_1}\szf_{h_2}$ since there always exists $k\in\Z$ such that
\[
\log(h_1h_2) = \log h_1 + \log h_2 + 2\pi i k \quad \mbox{on} \quad F
\]
for any continuous branches of the logarithms of $h_1$ and $h_2$. We shall refer to this property as to the multiplicativity property of Szeg\H{o} functions. It remains only to say that $\szf_h$ can be characterized as the unique holomorphic non-vanishing function in $D$ that assumes the value 1 at infinity, has bounded boundary values on $F$, and satisfies (\ref{eq:boundaryvalues}). Indeed, this follows from the uniqueness of $\oc\phi$ as the sectionally holomorphic function satisfying (\ref{eq:pls}), vanishing at infinity, and having bounded boundary values.

Employing the characterization described above, we can provide an explicit expression for the Szeg\H{o} function of a polynomial. Namely, let $E_n$ be a set of $2n$ not necessarily distinct nor finite points in $D$, $v_n$ be a polynomial with zeros at the finite points of $E_n$, and $r_n$ be the function associated to $E_n$ via (\ref{eq:rn}). Then we get from (\ref{eq:boundary}) and (\ref{eq:boundaryR}) that
\begin{equation}
\label{eq:szegopoly}
\szf_{v_n}^2 = \frac{1}{\gm_{v_n}}\frac{v_n}{r_n\map^{2n}}.
\end{equation}

Proceeding now with the investigation of the properties of Szeg\H{o} functions, we show that $\{\szf_{h_n}^\pm\}$ form relatively compact families in $C^*(F)$ whenever $\{h_n\}$ is a relatively compact family in $\dc^*_\omega(F)$. Without loss of generality we may assume that $\{h_n\}$ converges uniformly on $F$ to some non-vanishing function $h$. Fix an arbitrary branch of the logarithm of $h$. Clearly, there exists a choice of the branch of the logarithm of $h_n$ for each $n$ such that the sequence $\{\log h_n\}$ converges uniformly to $\log h$. Moreover, it follows from (\ref{eq:modcontlog}) that $\{\log h_n\}\subset\dc_\omega(F)$. Here, we slightly abuse the notation by using the same symbol $\omega$ despite the fact that the modulus of continuity that majorizes $\omega_{\log h_n}$ is different from the one that majorizes $\omega_{h_n}$. As only the existence of such a majorant is significant to our considerations, this should cause no ambiguity. Thus, $\{\phi_n\}\subset\dc_\omega(\Gamma)$, $\phi_n:=\log h_n \circ \jt$, and it converges to $\phi:=\log h \circ \jt$. As we saw in the previous section, it follows that $\{(\oc\phi_n)^\pm\}$ converge uniformly on $\Gamma$ to $(\oc\phi)^\pm$, respectively. Thus, the claim follows.

Finally, we compute the Szeg\H{o} function of $\hbar(\alpha,x;\cdot)$, $\alpha\in(0,1]$, $x\in F$, defined in (\ref{eq:hbar}). It is a simple observation that
\[
\gm_{\hbar(1,x;\cdot)} = 1 \quad \mbox{and} \quad \szf_{\hbar(1,x;\cdot)}(z) = \frac{2(z-x)}{\map(z)}, \quad z\in D,
\]
by uniqueness of the Szeg\H{o} function. Let us also recall that $\szf_{\hbar(1,x;\cdot)}$ is a non-vanishing holomorphic function in $D$, including at infinity, with zero winding number on any curve there, and  $\szf_{\hbar(1,x;\cdot)}(\infty)=1$. Thus, for any $\alpha\in(0,1]$, an $\alpha$-th power of $\szf_{\hbar(1,x;\cdot)}$ exists. In other words, the function
\[
\szf_\alpha:=\szf_{\hbar(1,x;\cdot)}^\alpha, \quad \szf_\alpha(\infty)= 1,
\]
is a single-valued holomorphic function in $D$ with continuous traces on each side of $F$.

Denote by $\ar_\map$ the branch of the argument of $\map$ which is continuous in $\C\setminus\{(-\infty,y]\cup F_{aux} \cup F\}$ and vanishes on positive reals greater than 1, where $F_{aux}$ was defined in Section \ref{subsec:strong2}. Then $\ar_\map$ has continuous extensions to $F$ that satisfy $\ar_\map^+ = -\ar_\map^-$ on $F$, $\ar_\map^+(1) = 0$, and $\ar_\map^+(-1) = \pi$. It can be easily checked that $\Arg\szf_\alpha = \alpha\left(\ar_x-\ar_\map\right)$, where $\Arg\szf_\alpha$ is the continuous branch of the argument of $\szf_\alpha$ that vanishes at infinity and $\ar_x$ was defined in Section \ref{subsec:strong2}. Thus,
\[
\szf_\alpha^\pm(t) = \frac{|2(t-x)|^\alpha}{|\map^\pm(t)|^\alpha}
\left\{
\begin{array}{ll}
\exp\left\{i\alpha(\ar_x(t)-\ar_\map^\pm(t))\right\}, & t\in F_{x,1}, \\
\exp\left\{i\alpha(\ar_x^\pm(t)-\ar_\map^\pm(t))\right\}, & t\in F_{-1,x}\setminus\{x\},
\end{array}
\right.
\]
and therefore $\szf_\alpha^+\szf_\alpha^- = \hbar(\alpha,x;\cdot)$ on $F$ by (\ref{eq:boundary}) and since $\ar_x^-=\ar_x^+-2\pi$ there. Altogether, we have proved that $\szf_{\hbar(\alpha,x;\cdot)} = \szf_\alpha$.

Now, define $\scf_\alpha := \szf_\alpha^+/\szf_\alpha^-$ on $F$. Then, for $x\neq\pm1$, we have
\begin{eqnarray}
\scf_\alpha(t) &=& |\map^-(t)|^{2\alpha}
\left\{
\begin{array}{ll}
\exp\left\{-i\alpha(\ar_\map^+(t)-\ar_\map^-(t))\right\} \\
\exp\left\{-i\alpha(-\ar_x^+(t)+\ar_x^-(t)+\ar_\map^+(t)-\ar_\map^-(t))\right\}
\end{array}
\right. \nonumber \\
{} &=& |\map^-(t)|^{2\alpha}\left\{
\begin{array}{ll}
\exp\left\{-2i\alpha\ar_\map^+(t)\right\}, & t\in F_{x,1}, \\
\exp\left\{-2i\alpha\ar_\map^+(t)+2i\alpha\pi\right\}, & t\in F_{-1,x}\setminus\{x\}
\end{array}
\right.; \nonumber
\end{eqnarray}
for $x=1$, we have that
\[
\scf_\alpha(t) = |\map^-(t)|^{2\alpha} \exp\left\{-2i\alpha\ar_\map^+(t)+2i\alpha\pi\right\}, \quad t\in F;
\]
and finally, for $x=-1$, we have that
\[
\scf_\alpha(t) = |\map^-(t)|^{2\alpha} \exp\left\{-2i\alpha\ar_\map^+(t)\right\}, \quad t\in F.
\]
Therefore, if $x\neq\pm1$ then $\scf_\alpha$ and $\scf_\alpha^{-1}$ are piecewise continuous functions on $F$, having jump-type singularities at $x$ of magnitudes $2\sin(\alpha\pi)|\map^-(x)/\map^+(x)|^{\alpha}$ and $2\sin(\alpha\pi)|\map^+(x)/\map^-(x)|^{\alpha}$, respectively, and $\scf_\alpha(\pm1)=1$; if $x$ is $\pm1$ then $\scf_\alpha$ is a continuous function on $F$, $\scf_\alpha(\mp1)=1$, and $\scf_\alpha(\pm1)=\exp\{\pm2i\alpha\pi\}$. Thus, (\ref{eq:scatjumps}) does indeed hold.

Now, let $F_0\subset F$ be a finite set of distinct points. Given $\alpha_x<1$ for every $x\in F_0$, let $\hbar$ be associated to $F_0$ via (\ref{eq:VF}). Then
\[
\szf_\hbar = \prod_{x\in F_0\setminus\{\pm1\}} \szf_{\hbar(\alpha_x,x;\cdot)} \prod_{x\in F_0\cap\{\pm1\}} \szf_{\hbar(\alpha_x/2,x;\cdot)}
\]
by the multiplicativity property of Szeg\H{o} functions. Moreover, $\scf_\hbar^+ = \szf_\hbar^+/\szf_\hbar^-$ and $\scf_\hbar^- = \szf_\hbar^-/\szf_\hbar^+$ are the products of $\scf_{\alpha_x}$ and $1/\scf_{\alpha_x}$, respectively, over $x\in F_0$, and therefore they are piecewise continuous function on $F$ with jump-type singularities at $x\in F_0\setminus\{\pm1\}$.

\subsection{Smirnov Classes}
\label{subsec:aux_smirnov}

Let $\Gamma$ be an analytic Jordan curve. Denote by $D^+$ and $D^-$ the bounded and unbounded components of $\overline\C\setminus\Gamma$.  Suppose also that $\Gamma$ is oriented counter-clockwise, i.e. $D^+$ lies on the left when $\Gamma$ is traversed in positive direction. Assume also that $0\in D^+$. Set $L^p$, $p\in[1,\infty)$, to be the space of $p$-summable functions on $\Gamma$ with respect to arclength and denote the corresponding norms by $\|\cdot\|_p$. For each $\phi\in L^1$ the Cauchy integral operator (\ref{eq:oc}) defines a sectionally holomorphic function in $D^+\cup D^-$. To be more specific we shall denote
\[
\oc^\pm\phi = (\oc\phi)_{|D^\pm}.
\]
Then each function $\oc^\pm\phi$ is holomorphic in the corresponding domain and has non-tangential boundary values a.e. on $\Gamma$, which, as before, we denote by $(\oc\phi)^\pm$. For each $p\in[1,\infty)$ we define the {\it Smirnov classes} $E_+^p$ and $E_-^p$ as
\begin{eqnarray}
E_+^p &:=& \left\{\phi\in L^p:~ \oc^-\phi\equiv0\right\}, \nonumber \\
E_-^p &:=& \left\{\phi\in L^p:~ \oc^+\phi\equiv(\oc^+\phi)(0)\right\}. \nonumber
\end{eqnarray}
Clearly, $E_\pm^p$ are closed subspaces of $L^p$. We note also that the Cauchy formula holds, i.e. if $\phi\in E^p_\pm$ then
\[
(\oc\phi)^+ = \phi \quad \mbox{and} \quad (\oc\phi)^- = -\phi + (\oc^+\phi)(0) \quad \mbox{a.e. on} \quad \Gamma
\]
by Sokhotski-Plemelj formula (\ref{eq:pls}). In other words, $\phi$ admits an analytic extension to $D^\pm$ whose $L^p$-means turn out to be uniformly bounded on the level curves of any conformal map from $\D$ onto $D^\pm$.  Conversely, if $h$ is a holomorphic function in $D^\pm$ whose $L^p$-means are uniformly bounded on some sequence of rectifiable Jordan curves in $D^\pm$ whose bounded or unbounded components of the complement exhaust $D^+$ or $D^-$, respectively, then the trace of $h$ belongs to $E_\pm^p$ (see \cite{Duren} and \cite{BottcherKarlovich} for details).

It will be of further convenience to define the projection operators from $L^p$ to the Smirnov classes $E_\pm^p$. To do that, we first define the Cauchy singular operator on $\Gamma$, say $\os$, by the rule
\[
(\os\phi)(t) := \frac{1}{\pi i}\int_\Gamma\frac{\phi(\tau)}{\tau-t}d\tau, \;\;\; t\in\Gamma,
\]
where the integral is understood in the sense of the principal value. We also set $\oi$ to stand for the identity operator on $L^p$ and we put $\dot E_-^p$, $p\in[1,\infty)$, to denote the subspace of $E_-^p$ consisting of functions vanishing at infinity, i.e.
\[
\dot E_-^p := \left\{\phi\in E_-^p:~(\oc^+\phi)(0)=0\right\} = \left\{\phi\in L^p:~\oc^+\phi\equiv0\right\}.
\]
In the considerations below we always assume that $p\in(1,\infty)$. Then
\begin{equation}
\label{eq:directdecomp}
L^p = E_+^p\oplus\dot E_-^p.
\end{equation}
We define complementary projections on $L^p$ in the following manner:
\begin{equation}
\label{eq:opoq}
\begin{array}{ll}
\op:~L^p \to E_+^p,      & \phi \mapsto (1/2)(\oi+\os)\phi,  \\
\oq:~L^p \to \dot E_-^p, & \phi \mapsto (1/2)(\oi-\os)\phi.
\end{array}
\end{equation}
These operators are bounded and can be used to rewrite the Sokhotski-Plemelj formula as
\begin{equation}
\label{eq:spl}
\op\phi = (\oc\phi)^+ \;\; \mbox{and} \;\; \oq\phi=-(\oc\phi)^-, \;\; \phi\in L^p.
\end{equation}

Finally, we characterize the weak convergence in $E_+^p$. According to the Riesz representation theorem, any linear functional $\ell\in(L^p)^*$ has a unique representation of the form
\begin{equation}
\label{eq:functionals}
\ell\phi = \int_\Gamma \phi f|d\tau|, \quad \phi\in E_+^p, \quad f\in L^q, \quad \frac1p+\frac1q=1.
\end{equation}
It follows now from the Hanh-Banach theorem that
\begin{equation}
\label{eq:linearfunctional1}
(E_+^p)^* \cong L^q/(E_+^p)^\perp,
\end{equation}
where the symbol ``$\cong$'' stands for ``isometrically isomorphic'' and $(E_+^p)^\perp$ is the annihilator of $E_+^p$ under the pairing (\ref{eq:functionals}). It can be easily checked that
\begin{equation}
\label{eq:linearfunctional2}
(E_+^p)^\perp = \left\{\psi\frac{d\tau}{|d\tau|}:~ \psi\in E_+^q\right\}.
\end{equation}
Since $\Gamma$ is analytic, multiplication by $d\tau/|d\tau|$ is an isometric isomorphism of $L^q$ into itself. Thus, by (\ref{eq:directdecomp}) and (\ref{eq:opoq}) every $f\in L^q$ can be uniquely decomposed as
\begin{equation}
\label{eq:linearfunctional3}
f = \op\left(f\frac{\overline{d\tau}}{|d\tau|}\right)\frac{d\tau}{|d\tau|} + \oq\left(f\frac{\overline{d\tau}}{|d\tau|}\right)\frac{d\tau}{|d\tau|}.
\end{equation}
Combining (\ref{eq:linearfunctional1}), (\ref{eq:linearfunctional2}), (\ref{eq:linearfunctional3}), and (\ref{eq:directdecomp}), we see that for any linear functional $\ell\in(E_+^p)^*$ there uniquely exists $\psi\in \dot E_-^q$ such that
\[
\ell\phi = \int_\Gamma \phi\psi d\tau, \;\;\; \phi\in E_+^p.
\]
Since $(\cdot-z)^{-1}\in \dot E_-^q$ for every $z\in D^+$, it follows from the Cauchy formula that a bounded sequence $\{\phi_n\}\subset E^p_+$ weakly converges to zero if and only if $\{\oc^+\phi_n\}$ converges to zero locally uniformly in $D^+$.

\subsection{Toeplitz and Hankel Operators}
\label{subsec:aux_th}

In the notation of the previous section, let $c\in L^\infty$, i.e. $c$ be a bounded function on $\Gamma$ with the norm $\|c\|_\infty$, and let $p\in(1,\infty)$. We define the {\it Toeplitz} and {\it Hankel} operators with symbol $c$ by the rule
\[
\begin{array}{ll}
\ot_c:~ E_+^p \to E_+^p,      & \phi \mapsto \op(c\phi) \nonumber, \\
\oh_c:~ E_+^p \to \dot E_-^p, & \phi \mapsto \oq(c\phi) \nonumber,
\end{array}
\]
respectively. Then $\ot_c$ and $\oh_c$ are bounded operators and $\oh_c$ is compact whenever $c\in C(\Gamma)$ (see, for instance, \cite[Prop. 6.21]{BottcherKarlovich}).

Let $\{\phi_n\}\subset E_+^p$ be a bounded sequence that weakly converges to zero and $\{c_n\}$ be a relatively compact family in $L^\infty$. Then for any limit point of $\{c_n\}$, say $c$, we have
\[
\|\ot_c\phi_n-\ot_{c_n}\phi_n\|_p = \|\op((c-c_n)\phi_n)\|_p \leq \|\op\|\|c-c_n\|_\infty\
\|\phi_n\|_p,
\]
where $\|\op\|$ is the operator norm of $\op$. Therefore, since $\op$ and $\ot_{c}$ are bounded operators, it can be readily verified that $\{\ot_{c_n}\phi_n\}$ weakly converges to zero. Moreover, using compactness of Hankel operators with continuous symbols, a similar argument shows that $\{\oh_{c_n}\phi_n\}$ strongly converges to zero in the $\|\cdot\|_p$-norm if $\{c_n\}$ is a relatively compact family in $C(\Gamma)$.

Let now $c$ be a continuous function on $\Gamma$ except for a finite set of points, say $\Gamma_0$. Assume further that $c$ possesses finite one-sided limits, say $c(\zeta^+)$ and $c(\zeta^-)$, at each $\zeta\in\Gamma_0$. Then there exists $c^\prime\in C(\Gamma)$ such that
\[
\|c-c^\prime\|_\infty = \frac12\max_{\zeta\in\Gamma_0}j_c(\zeta), \quad j_c(\zeta):=|c(\zeta^+)-c(\zeta^-)|.
\]
Thus, if $\{\phi_n\}\subset E_+^p$, $\|\phi_n\|=1$, is weakly convergent to zero, it holds that
\begin{eqnarray}
\limsup_{n\to\infty} \|\oh_c\phi_n\|_p &\leq& \limsup_{n\to\infty} \left(\|\oh_{c-c^\prime}\phi_n\|_p+\|\oh_{c^\prime}\phi_n\|_p\right) \nonumber \\
{} &\leq& \limsup_{n\to\infty} \|\oq((c-c^\prime)\phi_n)\|_p \leq \frac12\|\oq\|\max_{\zeta\in\Gamma_0}j_c(\zeta) \nonumber
\end{eqnarray}
since $\oh_{c^\prime}$ is a compact operator.

\subsection{On Conjugate Functions}
\label{subsec:aux_cf}

Let $\Gamma$ be an analytic Jordan curve and $D^+$ and $D^-$ stand, as before, for the bounded and unbounded components of the complement of $\Gamma$. Denote by $f$ a conformal map of the unit disk, $\D$, onto $D^+$. Further, let $\{\phi_n\}$ be a sequence of holomorphic functions in $D^+$ that continuously extend to $\Gamma$. Assume also that $\{|\phi_n|\}$ is a relatively compact family in $\dc_\omega^*(\Gamma)$. Define $\psi_n := \phi_n\circ f$. Then $\{\psi_n\}$ is a sequence of bounded holomorphic functions in $\D$ with continuous boundary values on the unit circle, $\T$. Moreover, since $f^\prime$ extends continuously on $\T$ by \cite[Thm. 3.3.5]{Pommerenke}, it holds that $\{|\psi_n|\}$ is a relatively compact family in $\dc_\omega^*(\T)$ (as before, we use the same symbol $\omega$ since only the existence of the majorant is important).

By the canonical factorization of Hardy functions in $\D$ (see \cite[Sec. 2.4]{Duren}), we get $\psi_n = \psi_n^i\psi_n^o$, where $|\psi_n^i|\equiv1$ on $\T$ and
\[
\psi_n^o(z) := \exp\left\{\int_\T\frac{\tau+z}{\tau-z}\log|\psi_n(\tau)|\frac{|d\tau|}{2\pi}\right\}, \quad z\in\D.
\]
In other words, $\psi_n^i$ and $\psi_n^o$ are the inner and outer factors of $\psi_n$, respectively. Since $|\psi_n^o|\equiv|\psi_n|$ on $\T$, $\{|\psi_n^o|\}$ is a relatively compact family in $\dc_\omega^*(\T)$. Moreover, since $|\psi^o_n|$ is uniformly bounded below on $\T$, we have that
\[
\omega_{\log|\psi^o_n|}(\delta) \leq \omega_{|\psi^o_n|}(\delta)/\min_{\tau\in\T}|\psi^o_n(\tau)| \leq k_1\omega(\delta)
\]
for all $\delta$ small enough, where $k_1$ is a positive constant independent of $n$. Thus, without loss of generality we may assume that $\{\log|\psi_n^o|\}\in\dc_\omega(\T)$. Denote by $\arg\psi_n^o$ the harmonic conjugate of $\log|\psi_n^o|$ normalized to be zero at zero. Then by \cite[Thm. III.1.3]{Garnett} we have that
\[
\omega_{\arg\psi_n^o}(\delta) \leq k_2 \left(\int_0^\delta\frac{\omega(t)}{t}dt + \delta\int_\delta^2\frac{\omega(t)}{t^2}dt\right),
\]
where $k_2$ is an absolute constant and the moduli of continuity of $\arg\psi_n^o$ are computed on $\T$. This, in particular, implies that $\{\arg\psi_n^o\}$ is a uniformly equicontinuous family of uniformly bounded functions on $\T$. Hence, by Ascoli-Arzel\`a theorem, $\{\arg\psi_n^o\}$ is relatively compact in $C(\T)$ and therefore so is $\{\psi_n^o\}$.

Summarizing the above we arrive at the following. If $\Gamma$ is an analytic Jordan curve and $\{\phi_n\}$ is a sequence of holomorphic functions in $D^+$ with continuous traces on $\Gamma$ such that $\{|\phi_n|\}$ forms a relatively compact family in $\dc_\omega^*(\Gamma)$, then each $\phi_n$ can be written as
\[
\phi_n = \phi_n^i\phi_n^o, \quad |\phi_n|\equiv|\phi_n^o| \quad \mbox{on} \quad \Gamma,
\]
and the sequence $\{\phi_n^o\}$ forms a relatively compact family in $C(\Gamma)$.

\section{Proofs}
\label{sec:proofs}

In the first part of this section we prove Theorems \ref{thm:sa1}, \ref{thm:sa2}, and \ref{thm:pade}. A number of places in these proofs contain references to the material of the previous section. The second part is devoted to the proof of Theorem \ref{thm:sp} and is independent of the rest of the paper.

\subsection{Proofs of Theorems \ref{thm:sa1}, \ref{thm:sa2}, and \ref{thm:pade}}
\label{subsec:proofs1}

\begin{proof}[Proof of Theorem \ref{thm:sa1}] The starting point for the proof of this theorem is the method of singular integral equation proposed by J. Nuttall in \cite{Nut90}.

By the very definition of $R_n$, (\ref{eq:wSokhotski}), (\ref{eq:boundary}), and (\ref{eq:szegodecomp}), it follows that
\begin{equation}
\label{eq:sokhotski}
(R_nw)^+ + (R_nw)^- = 2q_nw_n = 2\gm_{w_n}q_n~\frac{\szf_{w_n}^+\szf_{w_n}^-}{(\map^+)^n(\map^-)^n} \;\; \mbox{on} \;\; F.
\end{equation}
The last equation can be rewritten as
\begin{equation}
\label{eq:sokhotski1}
\left(\frac{R_nw\map^n}{\szf_{w_n}}\right)^\pm + \left(\frac{R_nw\map^n}{\szf_{w_n}}\right)^\mp \left(\frac{\map^n}{\szf_{w_n}}\right)^\pm \left(\frac{\szf_{w_n}}{\map^n}\right)^\mp = 2\gm_{w_n}\left(\frac{q_n\szf_{w_n}}{\map^n}\right)^\mp
\end{equation}
on $F$. Observe that
\begin{equation}
\label{eq:an}
a_n := \frac{R_nw\map^n}{\szf_{w_n}}
\end{equation}
is a holomorphic function in $D$ and $a_n^\pm\in\cs$, $a_n^-(\pm1)=a_n^+(\pm1)$. Indeed, $S_{w_n}$ is a non-vanishing and holomorphic in $D$, 
\[
w(z)\map^n(z)=O(z^{n+1}) \;\; \mbox{as} \;\; z\to\infty
\]
by the very definition of these functions, and
\[
R_n(z)=O(1/z^{n+1}) \;\; \mbox{as} \;\; z\to\infty
\]
due to orthogonality relations (\ref{eq:orthogonality}). The continuity of the boundary values $a_n^\pm$ and equality of $a_n^\pm$ at $\pm1$ follow from (\ref{eq:wSokhotski}) and (\ref{eq:endpoints}). Let also
\begin{equation}
\label{eq:bn}
b_n := 2\gm_{w_n} \frac{q_n\szf_{w_n}}{\map^n}.
\end{equation}
It is again immediate that $b_n$ is a holomorphic function in $D$, $b_n^\pm\in\cs$, and $b_n^+(\pm1)=b_n^-(\pm1)$. Then using the multiplicativity property of Szeg\H{o} functions, we can rewrite (\ref{eq:sokhotski1}) as
\begin{equation}
\label{eq:sokhotski2}
a_n^\pm + (a_nc_n)^\mp\frac{(\map^n\szf_{v_n})^\pm}{(\map^n\szf_{v_n})^\mp} = b_n^\mp \;\; \mbox{on} \;\; F,
\end{equation}
where
\begin{equation}
\label{eq:cn}
c_n^\mp:= \szf_{h_n}^\mp/\szf_{h_n}^\pm \quad \mbox{on} \quad F.
\end{equation}
Clearly, $c_n^\pm\in\cs$, $c_n^+c_n^- = 1$, and $c_n^+(\pm1)=c_n^-(\pm1)=1$. Moreover, $\{c_n^\pm\}$ are relatively compact families in $\cs$ (see Section \ref{subsec:aux_szf}, the discussion right after (\ref{eq:szegopoly})). Further, by (\ref{eq:szegodecomp}), (\ref{eq:boundary}), and (\ref{eq:szegopoly}) it holds that
\begin{equation}
\label{eq:rnmp}
\frac{(\map^n\szf_{v_n})^\pm}{(\map^n\szf_{v_n})^\mp} = \frac{1}{\gm_{v_n}}\frac{v_n}{(\map^{2n}\szf_{v_n}^2)^\mp} = r_n^\mp \;\; \mbox{on} \;\; F.
\end{equation}
Thus, we obtain from (\ref{eq:sokhotski2}) that
\begin{equation}
\label{eq:sokhotski3}
b_n^\mp = a_n^\pm + (a_nr_nc_n)^\mp \;\; \mbox{on} \;\; F.
\end{equation}

Let $\jt$ be the usual Joukovski transformation. Then $\Gamma:=\jt^{-1}(F)$ is an analytic curve (see Section \ref{subsec:aux_jouk}). Denote by $D^+$ and $D^-$ the bounded and unbounded components of the complement of $\Gamma$, respectively (we orient $\Gamma$ counter-clockwise), and define
\[
\hat f := f\circ\jt_{|D^+}, \;\;\; f=a_n,b_n,r_n. 
\]
Then $\hat a_n$, $\hat b_n$, $\hat r_n$ are holomorphic functions in $D^+$ with continuous traces
\begin{equation}
\label{eq:translate}
{\hat f}_{|\Gamma^\pm} = f^\mp\circ\jt, \;\;\; f=a_n, b_n, r_n, \;\;\; \Gamma^\pm := \map^\pm(F).
\end{equation}
Continuity of the boundary values immediately implies that these functions are bounded in $\overline{D^+}$ and therefore
\begin{equation}
\label{eq:hatfunctions}
\hat f\in E_+^p \;\; \mbox{and} \;\; \hat f(1/\cdot)\in E_-^p, \;\;\;  f=a_n,b_n,r_n,
\end{equation}
for any $p\in[1,\infty)$ by the characterization of Smirnov classes (Section \ref{subsec:aux_smirnov}). Moreover, the sequence $\{\hat r_n\}$ converges to zero locally uniformly in $D^+$ and therefore weakly goes to zero in $E_+^p$ (see Section \ref{subsec:aux_smirnov}) by the very definition of $\hat r_n$ and the fact that $F$ is symmetric with respect to $\E$. It follows from (\ref{eq:sokhotski3}) that
\begin{equation}
\label{eq:bvp}
\hat b_n(\tau) = \hat a_n(1/\tau) + (\hat a_n \hat r_n \hat c_n)(\tau), \;\;\; \tau\in\Gamma,
\end{equation}
where $\hat c_{n|\Gamma^\pm} := c_n^\mp\circ\jt$. Notice that $\hat c_n$ are, in fact, continuous functions on $\Gamma$ and $\{\hat c_n\}$ is a relatively compact family in $C(\Gamma)$. Fix an arbitrary $p\in(1,\infty)$ and denote
\begin{equation}
\label{eq:renormalizedan}
\hat a_{n,p} = \hat a_n/\|\hat a_n\|_p \; \mbox{on} \; D^+ \quad \mbox{and} \quad d_{n,p}(\cdot) := \hat a_{n,p}(0) - \hat a_{n,p}(1/\cdot) \; \mbox{on} \; D^-,
\end{equation}
where $\|\cdot\|_p$ is the $L^p$-norm on $\Gamma$. Clearly, $d_{n,p}\in\dot E_-^p$. Applying subsequently the operators $\op$ and $\oq$ (see (\ref{eq:opoq})) to both sides of (\ref{eq:bvp}), we get in the view of (\ref{eq:hatfunctions}) that
\begin{equation}
\label{eq:bvp1}
\left\{
\begin{array}{l}
\hat b_n/\|\hat a_n\|_p = \hat a_{n,p}(0) + \op(\hat a_{n,p} \hat r_n \hat c_n) \\
d_{n,p} = \oq(\hat a_{n,p} \hat r_n \hat c_n)
\end{array}
\right..
\end{equation}
Equivalently, we can rewrite (\ref{eq:bvp1}) as
\begin{equation}
\label{eq:bvp2}
\left\{
\begin{array}{l}
\hat b_n/\|\hat a_n\|_p = \hat a_{n,p}(0) - \ot_n(\hat a_{n,p} \hat r_n) \\
d_{n,p} = \oh_n(\hat a_{n,p} \hat r_n)
\end{array}
\right.,
\end{equation}
where $\ot_n$ and $\oh_n$ are the Toeplitz and Hankel operators with symbol $\hat c_n$, respectively (see Section \ref{subsec:aux_th}). By the normalization of $\{\hat a_{n,p}\}$ and the properties of $\{\hat r_n\}$, the norms $\|\hat a_{n,p}\hat r_n\|_p$ are uniformly bounded with $n$ and the sequence $\{\hat a_{n,p}\hat r_n\}$ weakly goes to zero. Thus,
\[
\|d_{n,p}\|_p \to 0 \;\; \mbox{as} \;\; n\to\infty
\]
by the compactness properties of Hankel operators with continuous symbols (see discussion in Section \ref{subsec:aux_th}). In particular, since
\[
\|d_{n,p}(1/\cdot)\|_p \leq \frac{\|d_{n,p}\|_p}{\min_{\tau\in\Gamma}|\tau|^{2/p}} \to 0 \quad \mbox{as} \quad n\to\infty
\]
(note that $\Gamma$ cannot pass through the origin since otherwise it would be unbounded as $\tau$ and $1/\tau$ belong to $\Gamma$ simultaneously), we have from (\ref{eq:renormalizedan}) that
\begin{equation}
\label{eq:anpat0}
\left\{
\begin{array}{lll}
|\hat a_n(0)|/\|\hat a_n\|_p &\to& 1 \\
\|1-\hat a_n/\hat a_n(0)\|_p &\to& 0
\end{array}
\right.
\quad \mbox{as} \quad n\to\infty.
\end{equation}
Using now properties of Toeplitz operators with $L^\infty$ symbols (see Section \ref{subsec:aux_th}), we derive from (\ref{eq:bvp2}), (\ref{eq:anpat0}), and the Cauchy formula that
\[
\left\{
\begin{array}{l}
\hat b_n/\hat a_n(0) = 1+o(1) \\
\hat a_n/\hat a_n(0) = 1+o(1)
\end{array}
\right.
\]
locally uniformly in $D^+$ or equivalently
\begin{equation}
\label{eq:theend}
\left\{
\begin{array}{lll}
\displaystyle \frac{2\gm_{w_n}}{\hat a_n(0)} \frac{q_n\szf_{w_n}}{\map^n} &=& 1+o(1) \\
\displaystyle \frac{1}{\hat a_n(0)} \frac{R_nw\map^n}{\szf_{w_n}} &=& 1+o(1)
\end{array}
\right.
\end{equation}
locally uniformly in $D$. The first asymptotic formula in (\ref{eq:theend}) immediately implies that $q_n$ has exact degree $n$ for all $n$ large enough and therefore can be normalized to be monic. Under such a normalization we deduce again from the first formula in (\ref{eq:theend}) by considering the point at infinity and using the normalization $\szf_{w_n}(\infty)=1$ that
\begin{equation}
\label{eq:an0}
\hat a_n(0) = [1+o(1)]2^{1-n}\gm_{w_n} = [1+o(1)] 2^n\gamma_n
\end{equation}
and (\ref{eq:sa1}) follows.

Now, let
\begin{equation}
\label{eq:dn}
d_n := \frac{R_nw}{\gamma_n\szf_n} -1 = \frac{a_n}{2^n\gamma_n}-1.
\end{equation}
Then $\hat d_n := d_n\circ\jt_{|D^+}$ is such that
\begin{equation}
\label{eq:dnhatto0}
\|\hat d_n\|_p \leq \|1-\hat a_n/\hat a_n(0)\|_p + o(1)\|\hat a_n\|_p/|\hat a_n(0)| \to 0 \;\; \mbox{as} \;\; n\to\infty
\end{equation}
for any $p\in(1,\infty)$ by (\ref{eq:an0}) and (\ref{eq:anpat0}), and therefore also for $p=1$ by H\"older's inequality. So, we have for any $p\in[1,\infty)$ that
\begin{eqnarray}
\|\hat d_n\|_p^p &=& \int_\Gamma |\hat d_n(\tau)|^p|d\tau| \nonumber \\
{} &=& \int_F\left(|d_n^-(t)|^p|(\map^+(t))^\prime| + |d_n^+(t)|^p|(\map^-(t))^\prime|\right)|dt| \nonumber \\
\label{eq:dnto0}
{} &=& \int_F\left(|d_n^-(t)|^p|\map^+(t)|+|d_n^+(t)|^p|\map^-(t)|\right) \left|\frac{dt}{w^+(t)}\right| \to 0
\end{eqnarray}
as $n\to\infty$. Since $|\map^\pm|\geq\min_\Gamma|\tau|$ and $|w^+(t)|=\sqrt{|1-t^2|}$, the second part of (\ref{eq:sa2}) follows from (\ref{eq:dn}) and (\ref{eq:dnto0}). Moreover, since we can rewrite (\ref{eq:sokhotski}) as
\[
q_n ~\gamma_n\szf_n^+\szf_n^- = (R_nw)^+ + (R_nw)^- = \gamma_n\left[((1+d_n)\szf_n)^+ + ((1+d_n)\szf_n)^-\right]
\]
by (\ref{eq:dn}), we get the first part of (\ref{eq:sa2}).

So, it remains only to prove (\ref{eq:sa3}). By Lavrentiev's theorem, which is the analog of Weierstrass theorem for complex-valued functions on an arc, it is enough to show that
\begin{equation}
\label{eq:mergelyan}
\int_F t^j \gamma_n^{-1}q_n^2(t)w_n(t)\frac{dt}{w^+(t)} \to \int_F t^j \frac{dt}{w^+(t)}, \;\;\; j=0,1,\ldots ~.
\end{equation}
By (\ref{eq:sokhotski}), we get
\begin{eqnarray}
\gamma_n^{-1}q_n^2w_n &=& \frac{1}{4\gamma_nw_n}\left[ \left((R_nw)^+\right)^2 + 2(R_nw)^+(R_nw)^- + \left((R_nw)^-\right)^2 \right] \nonumber \\
{} &=& \gamma_n^{-2}\left[ \frac12\frac{\left((R_nw)^+\right)^2}{\szf_n^+\szf_n^-} + \frac{(R_nw)^+(R_nw)^-}{\szf_n^+\szf_n^-} + \frac12\frac{\left((R_nw)^-\right)^2}{\szf_n^+\szf_n^-} \right]. \nonumber
\end{eqnarray}
Recalling the definitions of $a_n$, $r_n$, $c_n$, and $d_n$, we get in view of (\ref{eq:rnmp}) that
\begin{eqnarray}
\gamma_n^{-1}q_n^2w_n &=& 2^{-2n}\gamma_n^{-2}\left[(a_n^2r_nc_n)^+/2 + a_n^+a_n^- + (a_n^2r_nc_n)^-/2 \right] \nonumber \\
{} &=& 1 + \left[d_n^- + d_n^+ + d_n^-d_n^+\right] \nonumber \\
{} && + \left[((1+d_n)^2r_nc_n)^+ + ((1+d_n)^2r_nc_n)^-\right]/2. \nonumber
\end{eqnarray}
Thus,
\[
\int_F t^j \frac{q_n^2(t)w_n(t)}{\gamma_nw^+(t)}dt = \int_Ft^j\frac{dt}{w^+(t)} + I_{1,n} - \pi i I_{2,n},
\]
where
\[
I_{1,n} := \int_F \left[d_n^-(t) + d_n^+(t) + d_n^-(t)d_n^+(t)\right] \frac{t^jdt}{w^+(t)}
\]
and
\[
I_{2,n} := -\frac{1}{2\pi i}\int_F \left[(1+d_n^+(t))^2(r_n^+c_n^+)(t) + (1+d_n^-(t))^2(r_n^-c_n^-)(t)\right]\frac{t^jdt}{w^+(t)}.
\]
Equations (\ref{eq:mergelyan}) shall follow upon showing that $I_{1,n}$ and $I_{2,n}$ converge to zero as $n$ grows large. Clearly, this holds for $I_{1,n}$ by (\ref{eq:dnto0}) since $|\map^\pm|$ are bounded on $F$. Further, substituting $t=\jt(\tau)$, $\tau\in\Gamma$, we get as in (\ref{eq:translation}) that
\[
\int_F (1+d_n^\pm(t))^2(r_nc_n)^\pm(t)\frac{t^jdt}{w^+(t)} = -\int_{\Gamma^\mp} (1+\hat d_n(\tau))^2(\hat r_n\hat c_n)(\tau) \jt^j(\tau) \frac{d\tau}{\tau},
\]
where we integrate on $\Gamma^-$ from $-1$ to 1 and on $\Gamma^+$ from 1 to $-1$. Thus, by the very definition of the operators $\oc$ and $\op$ (see (\ref{eq:oc}) and (\ref{eq:opoq})), it holds that
\begin{eqnarray}
I_{2,n} &=& \left(\oc^+((1+\hat d_n)^2\hat r_n\hat c_n\jt^j)\right)(0) = \left(\oc^+\op((1+\hat d_n)^2\hat r_n\hat c_n\jt^j)\right)(0) \nonumber \\
{} &=& \left(\oc^+\ot_{j,n}((1+\hat d_n)^2\hat r_n)\right)(0), \nonumber
\end{eqnarray}
where $\ot_{j,n}$ is the Toeplitz operator with symbol $\hat c_n\jt^j$. Since $\|1+\hat d_n\|_p^2 = 1+o(1)$ for any $p\in(1,\infty)$ by (\ref{eq:dnhatto0}) and since $|\hat r_n|$ are uniformly bounded above on $\Gamma$ while $\{\hat r_n\}$ converges to zero locally uniformly in $D^+$ by the assumptions of the theorem, the sequence $\{(1+\hat d_n)^2\hat r_n\}$ weakly goes to zero in each $L^p$, $p\in(1,\infty)$. Furthermore, $\{\hat c_n\jt^j\}$ is a relatively compact family in $C(\Gamma)$ for every $j=0,1,\ldots$. Therefore $\ot_{j,n}((1+\hat d_n)^2\hat r_n)$ converges weakly to zero in any $L^p$, $p\in(1,\infty)$, and subsequently $I_{2,n}$ tend to zero as $n$ grows large (see Section \ref{subsec:aux_th}). This finishes the proof of the theorem.
\end{proof}
\begin{proof}[Proof of Theorem \ref{thm:sa2}] The following is a modification of the proof of Theorem \ref{thm:sa1}. Thus, we shall keep all the notation we used in that proof.

Set $\hat c_{|\Gamma^\pm}:=c^\mp_\hbar\circ\jt$, where $c^\pm_\hbar$ were defined after (\ref{eq:VF}). Repeating the steps leading to (\ref{eq:bvp}), we get that
\begin{equation}
\label{eq:plemelj}
\hat b_n(\tau) = \hat a_n(1/\tau) + (\hat a_n\hat r_n\hat c_n\hat c)(\tau), \quad  \tau\in\Gamma.
\end{equation}
As before, $\hat b_n$ and $\hat r_n$ are holomorphic functions in $D^+$ with continuous traces on $\Gamma$, $\{\hat r_n\}$ converges to zero locally uniformly in $D^+$, and the boundary values of $|\hat r_n|$ are uniformly bounded above on $\Gamma$. Further, $\{\hat c_n\}$ is a relatively compact family in $C(\Gamma)$ and $\hat a_n$ are holomorphic functions in $D^+$. However, the traces of $\hat a_n$ are continuous only on $\Gamma\setminus\Gamma_0$, $\Gamma_0=\jt^{-1}(F_0)$. Moreover, at each $\zeta\in\Gamma_0$, each $\hat a_n$ has an algebraic singularity of order $-\alpha_x > -1/2$, $x=\jt(\zeta)$. Thus, by the characterization of Smirnov classes (Section \ref{subsec:aux_smirnov}), we have that $a_n\in E_+^2$. Finally, the  function $\hat c$ is piecewise continuous on $\Gamma$, i.e. $\hat c$ is continuous on $\Gamma\setminus\Gamma_0$ with jump-type discontinuities at each $\zeta\in\Gamma_0$.

Given (\ref{eq:plemelj}), we obtain as in (\ref{eq:bvp1}) that
\begin{equation}
\label{eq:newbvp}
\left\{
\begin{array}{l}
\hat b_n/\|\hat a_n\|_2 = \hat a_{n,2}(0) + \op(\hat a_{n,2} \hat r_n \hat c_n \hat c) \\
d_{n,2} = \oq(\hat a_{n,2} \hat r_n \hat c_n \hat c)
\end{array}
\right..
\end{equation}
By assumptions on $\E$, it holds that $\{|\hat r_n|\}$ is a relatively compact family in $\dc^*_\omega(\Gamma)$. Thus, as explained in Section \ref{subsec:aux_cf}, we have that $\hat r_n = \hat r_n^i \hat r_n^o$, where $\{\hat r_n^o\}$ is a normal family in $D^+$ with continuous boundary values on $\Gamma$ that form a relatively compact family in $C(\Gamma)$. Moreover, $|\hat r_n|=|\hat r_n^o|$ on $\Gamma$ and since neither of the limit points of $\{\hat r_n^o\}$ can vanish in $D^+$ by the outer character of $\hat r_n^o$, it follows from the assumptions on $r_n$ that $\{\hat r_n^i\}$ converges to zero locally uniformly in $D^+$   and $|\hat r_n^i|\equiv1$ on $\Gamma$.

Now, we can write (\ref{eq:newbvp}) as
\begin{equation}
\label{eq:newbvp1}
\left\{
\begin{array}{l}
\hat b_n/\|\hat a_n\|_2 = \hat a_{n,2}(0) - \ot_n(\hat a_{n,2} \hat r_n) \\
d_{n,2} = \oh_n(\hat a_{n,2} \hat r_n^i)
\end{array}
\right.,
\end{equation}
where $\ot_n$ is the Toeplitz operator with symbol $\hat c_n\hat c$ and $\oh_n$ is the Hankel operator with symbol $\hat c_n\hat r_n^o\hat c$. Then we get from the discussion at the end of Section \ref{subsec:aux_th} that
\[
\limsup \|d_{n,2}\|_2 \leq \frac12\|\oq\| \max_{\zeta\in\Gamma_0} \limsup_{n\to\infty} \left|(\hat c_n\hat r_n^o\hat c)(\zeta^+) - (\hat c_n\hat r_n^o\hat c)(\zeta^-)\right|
\]
since $\{\hat c_n\hat r_n^o\}$ is a relatively compact family in $C(\Gamma)$, $\|\hat a_{n,2}\hat r_n^i\|_2=1$, and $\{\hat a_{n,2}\hat r_n^i\}$ weakly converges to zero in $E_+^2$. Then it immediately follows from the very definition of $\hat c_n$, $\hat c$, $\hat r_n$, and (\ref{eq:upsilon}) that
\[
\limsup_{n\to\infty}\|d_{n,2}(1/\cdot)\|_2 < 1
\]
and since $d_{n,2}(1/\cdot)=(\hat a_n(0)-\hat a_n(\cdot))/\|\hat a_n\|_2$, we derive that
\begin{equation}
\label{eq:anfinite}
\limsup_{n\to\infty} \|\hat a_n\|_2/|\hat a_n(0)| < \infty.
\end{equation}
As $\ot_n$ and $\oh_n$ are bounded operators acting on weakly convergent sequences, we get from (\ref{eq:newbvp1}) and (\ref{eq:anfinite}) that
\begin{equation}
\label{eq:againtheend}
\left\{
\begin{array}{l}
\hat b_n/\hat a_n(0) = 1+o(1) \\
\hat a_n/\hat a_n(0) = 1+o(1)
\end{array}
\right.
\end{equation}
locally uniformly in $D^+$. As in the proof of Theorem \ref{thm:sa1}, (\ref{eq:againtheend}) implies (\ref{eq:theend}) and (\ref{eq:an0}) and therefore (\ref{eq:sa1}) follows under the present assumptions.

Define $d_n$ as in (\ref{eq:dn}). Then it still holds that
\[
\hat d_n = d_n\circ\jt_{|D^+} = [1 + o(1)]\frac{\hat a_n}{\hat a_n(0)} - 1
\]
and hence $\limsup_{n\to\infty}\|\hat d_n\|_2 < \infty$ by (\ref{eq:anfinite}). Thus, as in (\ref{eq:dnto0}), we get
\[
\limsup_{n\to\infty} \int_F\left(|d_n^-(t)|^2|\map^+(t)|+|d_n^+(t)|^2|\map^-(t)|\right) \left|\frac{dt}{w^+(t)}\right| < \infty.
\]
The rest of the conclusions of the theorem now follow as in the proof of Theorem \nolinebreak \ref{thm:sa1} from the above estimate.
\end{proof}
\begin{proof}[Proof of Theorem \ref{thm:pade}] The existence of schemes that make $F$ symmetric with respect to them was shown in Theorem \ref{thm:sp}. It is also well-known and follows easily from the defining properties of Pad\'e approximants, that the denominators of $\Pi_n$ satisfy non-Hermitian orthogonality relation of the form
\[
\int t^jq_n(t)\frac{d\mu(t)}{v_n(t)} = 0, \;\;\; j=0,\ldots,n-1,
\]
and the error of approximation is given by
\[
(f_\mu-\Pi_n)(z) = \frac{v_n(z)}{q_n^2(z)} \int \frac{q_n^2(t)}{v_n(t)}\frac{d\mu(t)}{z-t}, \quad z\in\overline\C\setminus\supp(\mu).
\]
Given (\ref{eq:measuremu}), we see that the asymptotic behavior of $q_n$ is governed by (\ref{eq:sa1}). Thus, using the orthogonality of $q_n$, it is not difficult to see that we can write
\[
f_\mu-\Pi_n = \frac{v_nR_n}{q_n} \quad \mbox{in} \quad D=\overline\C\setminus\supp(\mu),
\]
where $R_n$ is the function of the second kind associated to $q_n$ via (\ref{eq:secondkind}). Hence, equations (\ref{eq:sa1}) imply that
\begin{eqnarray}
f_\mu-\Pi_n &=& (1+o(1))~\frac{\gamma_nv_n\szf_n^2}{w} = \frac{2\gm_{\dot\mu}+o(1)}{w}~\frac{\szf_{\dot\mu}^2~v_n}{\gm_{v_n}\map^{2n}\szf_{v_n}^2} \nonumber \\
{} &=& \frac{2\gm_{\dot\mu}+o(1)}{w}~\szf_{\dot\mu}^2~r_n \nonumber
\end{eqnarray}
locally uniformly in $D$, where we used (\ref{eq:szegopoly}).
\end{proof}

\subsection{Proof of Theorem \ref{thm:sp}}
\label{subsec:proofs2}

Before we proceed with the main part of the proof, we claim that our assumptions on $F$ yield the existence of the partial derivatives $\partial U^{\lambda-\nu}/\partial n^\pm$ almost everywhere on $F$, where $\lambda$ is the weighted equilibrium distribution on $F$ in the field $-U^\nu$ and $\nu$ is a compactly supported probability Borel measure in $D=\overline\C\setminus F$.

Let $\psi$ be the conformal map of $D$ into the unit disk, $\D$, such that $\psi(\infty)=0$ and $\psi^\prime(\infty)>0$. Put
\begin{equation}
\label{eq:generalBlaschke}
b(e;z) := \frac{\psi(z)-\psi(e)}{1-\overline{\psi(e)}\psi(z)}, \quad e,z\in D.
\end{equation}
Then $b(e;\cdot)$ is a holomorphic function in $D$ with a simple zero at $e$ and unimodular boundary values on $F$. Hence, the \emph{Green function for $D$ with pole at $e\in D$} \cite[Sec. II.4]{SaffTotik} is given by $-\log|b(e;\cdot)|$ and the \emph{Green potential} of a Borel measure $\nu$, $\supp(\nu)\subset D$, not necessarily compact, is given by
\begin{equation}
\label{eq:greenpotential}
 U_D^\nu(z) := -\int\log|b(t;z)|d\nu(t), \quad z\in D.
\end{equation}
Furthermore, since $D$ is simply connected and therefore regular with respect to the Dirichlet problem, $U_D^\nu$ extends continuously to $F$ and is identically zero there.

Assume now that $\nu$ has compact support. Then it is an easy consequence of the characterization of the weighted equilibrium distribution \cite[Thm. I.3.3]{SaffTotik} and the representation of Green potentials via logarithmic potentials \cite[Thm. II.5.1]{SaffTotik} that $\lambda$ as above is, in fact, the {\it balayage} \cite[Sec. II.4]{SaffTotik} of $\nu$ onto $F$, $\supp(\lambda)=F$, and
\begin{equation}
\label{eq:balayage}
(U^\lambda-U^\nu)(z) = -U_D^\nu(z) - \int\log|b(\infty;t)|d\nu(t), \quad z\in \overline\C.
\end{equation}
Thus, the existence of $\partial U^{\lambda-\nu}/\partial n^\pm$ specializes to the existence of $\partial U_D^\nu/\partial n^\pm$.

Let now $\jt$ be the Joukovski transformation and $\Gamma:=\jt^{-1}(F)$. It was explained in Section \ref{subsec:aux_jouk} that $\Gamma$ is a Jordan curve. Denote by $D^+$ and $D^-$ the bounded and unbounded components of the complement of $\Gamma$, respectively. Then
\[
\phi := \psi\circ\jt, \quad \phi:D^+\to\D, \quad \phi(0) = 0, \quad \mbox{and} \quad \phi^\prime(0)>0,
\]
is a conformal map. Set
\begin{equation}
\label{eq:uphi}
u(z) := -(U_D^\nu\circ\jt)(z) = \int \log \left| \frac{\phi(z)-\phi(\tau)}{1-\phi(z)\overline{\phi(\tau)}} \right| d\nu^+(\tau), \quad z\in D^+,
\end{equation}
where $d\nu^+:=d(\nu\circ\jt_{|D^+})$. As $\jt$ is conformal on $\Gamma\setminus\{\pm1\}$, the existence of $\partial u/\partial n$ a.e. on $\Gamma$, where $\partial/\partial n$ is a partial derivative with respect to the inner normal on $\Gamma$, will imply the existence of $\partial U_D^\nu/\partial n^\pm$ a.e. on $F$. As
\[
\frac{\partial u}{\partial n}(\tau) := \lim_{\delta\to0}\left\langle \vec{n}_\tau,\nabla u(\tau+\delta n_\tau)\right\rangle = 2\lim_{\delta\to0}\re\left(n_\tau\frac{\partial u}{\partial z}(\tau+\delta n_\tau)\right),
\]
where $\langle\cdot,\cdot\rangle$ is the usual scalar product in $\R^2$, $\vec{n}_\tau$ is the inner normal vector at $\tau$, $n_\tau$ is the unimodular complex number corresponding to $\vec{n}_\tau$, and $(\partial/\partial z):=((\partial/\partial x)-i(\partial/\partial y))/2$, it is needed to show that the function $\partial u/\partial z$, holomorphic in $D^+\setminus\supp(\nu^+)$, has non-tangential limits a.e. on $\Gamma$. It can be readily verified that
\begin{equation}
\label{eq:duphi}
\frac{\partial u}{\partial z}(z) = \frac{\phi^\prime(z)}{2} \int\frac{1-|\phi(\tau)|^2}{(1-\overline{\phi(\tau)}\phi(z))(\phi(z)-\phi(\tau))}d\nu^+(\tau), \quad z\in D^+.
\end{equation}
Moreover, by construction, we have that $\phi^\prime\in E_+^1$ (see Section \ref{subsec:aux_smirnov} for the definition of the Smirnov class $E_+^1$). Recall (see Section \ref{subsec:aux_jouk}) also that the assumed condition on the behavior of $F$ near the endpoints implies rectifiability of $\Gamma$. Hence, any function in $E_+^1$ has non-tangential limits a.e. on $\Gamma$ by \cite[Thm. 10.3]{Duren}. As $\phi$ extends continuously on $\Gamma$ by Carath\'eodory's theorem \cite[Sec. 2.1]{Pommerenke}, the desired claim on existence of the boundary values of $\partial u/\partial z$ and respectively normal derivatives $\partial U^{\lambda-\nu}/\partial n^\pm$ follows.
\smallskip
\newline
{\bf (a)$\Rightarrow$(b)}: Let $F$ be symmetric with respect to a triangular scheme of points $\E=\{E_n\}$. Denote by $\nu_n$ the counting measures of points in $E_n$ and let $\nu$ be a weak$^*$ limit point of $\{\nu_n\}$. In other words, there exists  a subsequence $\N_1\subset\N$ such that $\nu_n \cws \nu$, $\N_1\ni n\to\infty$. If $\nu$ has bounded support, we shall show that $F$ is symmetric in the field $-U^\nu$. By (\ref{eq:SProperty}) and (\ref{eq:balayage}), this is equivalent to proving that
\begin{equation}
\label{eq:step1}
\frac{\partial U_D^\nu}{\partial n^+} = \frac{\partial U_D^\nu}{\partial n^-} \quad \mbox{a .e. on} \quad F.
\end{equation}
If $\nu$ has unbounded support, we claim that there exists a compactly supported Borel measure $\nu^*$ such that $U_D^{\nu^*}=U_D^\nu$ in some open neighborhood of $F$. Thus, the equality in (\ref{eq:step1}) will be sufficient to show that $F$ is symmetric in the field $-U^{\nu^*}$. To construct $\nu^*$, set $\kappa(z)=1/(z-a)$, $a\notin\supp(\nu)\cup F$. Clearly, $U_D^\nu = U_{\kappa(D)}^{\nu_\kappa}\circ\kappa$, where $d\nu_\kappa=d(\nu\circ\kappa^{-1})$ has compact support. Let $L$ be a Jordan curve in $\kappa(D)$ encompassing $\supp(\nu_\kappa)$. Then the balayage of $\nu_\kappa$ on $L$, say $\nu_\kappa^*$, is such that $U_{\kappa(D)}^{\nu_\kappa} = U_{\kappa(D)}^{\nu_\kappa^*}$ outside of $L$ \cite[Thm. II.4.1]{SaffTotik}. Thus, the claim follows for $d(\nu^*):=d(\nu^*_\kappa\circ\kappa)$.

In order to prove that the equality in (\ref{eq:step1}) indeed holds, we derive different from (\ref{eq:greenpotential}) integral representation for the Green potential of the above measure $\nu$. It follows from (\ref{eq:greenpotential}) that
\[
U_D^{\nu_n}(z) = -\frac{1}{2n}\log |b_n(z)|, \quad b_n(z) := \prod_{e\in E_n} b(e;z), \quad z\in D, \quad n\in\N_1.
\]
Further, since $F$ is symmetric with respect to $\E$, $\{\log|r_n/b_n|\}$ is a sequence of harmonic functions in $D$ whose boundary values uniformly bounded above and away from zero. Hence, we get from the maximum principle for harmonic functions that
\begin{equation}
\label{eq:maxprin}
\frac{1}{2n}\left(\log |r_n(z)| - \log|b_n(z)|\right) \to 0 \quad \mbox{uniformly in} \quad \overline\C.
\end{equation}
Moreover, removing $o(n)$ factors simultaneously from $r_n$ and $b_n$ will not alter this conclusion. Indeed, new difference will remain harmonic and its boundary values on $F$ will still uniformly go to zero since $|b^\pm(t;\cdot)|\equiv1$ and $|r^\pm(t;\cdot)|\leq\const$ for all $t\in\supp(\E)$, where $\const$ depends only on $\supp(\E)$. On the other hand, since $\nu_n$ weakly converge to $\nu$, it is possible to write each $\nu_n$ as a sum of two measures, say $\nu_{n,1}$ and $\nu_{n,2}$, such that $|\nu_{n,2}|=o(n)$ and the supports of $\nu_{n,1}$ are asymptotically contained in any open set around $\supp(\nu)$. Let $r_{n,1}$ and $b_{n,1}$ correspond to $\nu_{n,1}$ as $r_n$ and $b_n$ correspond to $\nu_n$. Then we obtain from the weak$^*$ convergence of $\nu_{n,1}$ to $\nu$ and the remark after (\ref{eq:maxprin}) that
\[
U_D^\nu = \lim_{\N_1\ni n\to\infty} U_D^{\nu_{n,1}} = \lim_{\N_1\ni n\to\infty} -\frac{1}{2n}\log|r_{n,1}|,
\]
locally uniformly in $D\setminus\supp(\nu)$. So, writing $(1/2n)\log|r_{n,1}|$ as an integral of $\log|r(\cdot;z)|$ against $\nu_{n,1}$, we derive from the weak$^*$ convergence of measures that
\begin{equation}
\label{eq:newkernel}
U_D^\nu(z) = -\int\log|r(t;z)|d\nu(t), \quad z\in D\setminus\supp(\nu).
\end{equation}

We proceed by reformulating (\ref{eq:step1}) in terms of the boundary values of the complex-valued function $H:=\partial U_D^\nu/\partial z$. Since $U_D^\nu$ is harmonic in $D\setminus\supp(\nu)$, $H$ is holomorphic there on account of Cauchy-Riemann equations \cite[Sec. 4.6.1]{Ahlfors2}. Moreover, it follows from (\ref{eq:newkernel}) that
\begin{equation}
\label{eq:repH}
H(z) = -\frac{\map^\prime(z)}{2} \int\frac{1-\map^2(\tau)}{(1-\map(\tau)\map(z))(\map(z)-\map(\tau))}d\nu(\tau),
\end{equation}
$z\in D\setminus\supp(\nu)$. Observe further that $H$ extends analytically across each side of $F\setminus\{\pm1\}$. (Note that $\map^\prime=\map/w$ and therefore it extends analytically across each side of $F\setminus\{\pm1\}$ as $\map$ and $w$ obviously do.)

Let now $t\in F_\tau$, where $F_\tau\subset F\setminus\{\pm1\}$ is the set of points at which $F$ possesses tangents. Denote by $\vec{\tau_t}$ and $\vec{n^\pm_t}$ the unit tangent vector and the one-sided unit normal vectors to $F$ at $t$. Further, put $\tau_t$ and $n_t^\pm$ to be the corresponding unimodular complex numbers, $n^+_t = i\tau_t$. Then
\[
\langle\nabla U_D^\nu(z),\vec{n^\pm_t}\rangle = 2\re\left(n_t^\pm ~ H(z)\right), \quad z\in D.
\]
As $H$ extends holomorphically across $F\setminus\{\pm1\}$, the above equality also holds at $z=t$. In other words, we have that
\begin{equation}
\label{eq:step2}
\frac{\partial U_D^\nu}{\partial n^\pm}(t) = 2\re\left(n^\pm_t H^\pm(t)\right) = 2n^\pm_t H^\pm(t), \quad t\in F_\tau.
\end{equation}
The last equality in (\ref{eq:step2}) is valid because $U_D^\nu\equiv0$ on $F$ and therefore
\[
0 = \frac{\partial U_D^\nu}{\partial \tau}(t) = \mp2\im\left(n^\pm_t H^\pm(t)\right), \quad t\in F_\tau.
\]
As $n^+_t=-n^-_t$, (\ref{eq:step1}) will follow from (\ref{eq:step2}) if we show that
\begin{equation}
\label{eq:step3}
H^+ = - H^- \quad \mbox{on} \quad F_\tau.
\end{equation}
The latter is well-understood in the theory of symmetric contours \cite[pg. 335]{GRakh87} and can be seen as follows.

Observe that $(\map^\pm)^\prime = (\map^\prime)^\pm$ on $F\setminus\{\pm1\}$ since $\map$ extends holomorphically across each side of $F\setminus\{\pm1\}$. Therefore, we have by (\ref{eq:boundary}) that
\[
(\map^-)^\prime = \frac{\map^-}{w^-} = \frac{-1}{(\map^+)^2}\frac{\map^+}{w^+} = -\frac{(\map^+)^\prime}{(\map^+)^2} \quad \mbox{on} \quad F\setminus\{\pm1\}.
\]
Then by (\ref{eq:repH}) we have for the unrestricted boundary values of $H$ on each side of $F$ that
\begin{eqnarray}
H^+(t) &=& -\frac{(\map^+)^\prime(t)}{2} \int \frac{1-\map^2(\tau)}{(1-\map(\tau)\map^+(t))(\map^+(t)-\map(\tau))}d\nu(\tau) \nonumber \\
{} &=& -\frac{(\map^+)^\prime(t)}{2(\map^+)^2(t)} \int \frac{1-\map^2(\tau)}{(\map^-(t)-\map(\tau))(1-\map(\tau)\map^-(t))}d\nu(\tau) = -H^-(t), \nonumber
\end{eqnarray}
which finishes this part of the proof.
\smallskip
\newline
{\bf (b)$\Rightarrow$(c)}: Let $F$ be symmetric in the field $-U^\nu$, where $\nu$ is a positive Borel measure compactly supported in $D$. We show that there exists a univalent function $p$ holomorphic in some neighborhood of $[-1,1]$ such that $F=p([-1,1])$.

Let $\Gamma$, $D^+$, and $D^-$ be as before and denote by $\Omega$ the domain in $\overline\C$ such that $\jt(\Omega)=\overline\C\setminus\supp(\nu)$. Set
\[
u(z) = \mp U_D^\nu(\jt(z)), \quad z\in\Omega\cap D^\pm.
\]
Then $u$ is identically zero on $\Gamma$, harmonic and negative in $D^+$, and harmonic and positive in $D^-$. Moreover, $u_{|D^+}$ has integral representation (\ref{eq:uphi}) and, as explained after (\ref{eq:duphi}), it has well-defined non-tangential boundary values, say $(\partial u/\partial z)^+$, that are integrable on $\Gamma$. Since $u(1/z)=-u(z)$ and $D^-=\{1/z:~z\in D^+\}$, the trace $(\partial u/\partial z)^-$ also exists and belongs to $L^1$.

In another connection, it is easy to verify that
\begin{equation}
\label{eq:uhj}
(\partial u/\partial z)(z) = \mp H(\jt(z))\jt^\prime(z), \quad z\in\Omega\cap D^\pm, \quad H:=\partial U_D^\nu/\partial z,
\end{equation}
is a sectionally holomorphic function. Let $\map$ be given by (\ref{eq:map}). Then $\Gamma$ is the boundary of $\map(D)$ and, as before, we set $\Gamma^\pm=\map^\pm(F)$. Since (\ref{eq:SProperty}) is equivalent to (\ref{eq:step3}) as explained in the previous part of the proof, we obtain for a.e. $\tau\in\Gamma^-$ that
\[
\left(\partial u/\partial z\right)^+(\tau) = -H^+(\jt(\tau))\jt^\prime(\tau) = H^-(\jt(\tau))\jt^\prime(\tau) = \left(\partial u/\partial z\right)^-(\tau).
\]
Analogously, we can show that $(\partial u/\partial z)^+$ coinsides with $(\partial u/\partial z)^-$ a.e. on $\Gamma^+$ and therefore a.e. on $\Gamma$. It follows now from an application of the Cauchy formula that $\partial u/\partial z$ is analytic across $\Gamma$, i.e. $\partial u/\partial z$ is a holomorphic function in the whole $\Omega$.

Recall that $u$ is identically zero on $\Gamma$ and does not vanish anywhere else in $\overline\C$. Hence, the level lines $\Gamma_\epsilon:=\{z:~u(z)=\epsilon\}$ are single Jordan curves for all $\epsilon$ sufficiently close to zero. Let $L$ be such a level line in $D^+$ and $\Omega_L$ be the annular domain bounded by $L$ and $L^{-1}:=\{z:~1/z\in L\}$. Observed that by Sard's theorem on regular values of smooth functions we always can assume that $L$ is smooth. Since $u$ is constant on $L$, the tangential derivative of $u$ there is zero and we get as in (\ref{eq:step2}) that
\begin{equation}
\label{eq:normalonL}
\frac{\partial u}{\partial n}(\tau) = \langle \nabla u(\tau), \vec{n}_\tau \rangle = 2n_\tau\frac{\partial u}{\partial z}(\tau), \quad \tau\in L,
\end{equation}
where $\partial/\partial n$ is the derivative in the direction of the inner normal with respect to $\Omega_L$, $\vec{n}_\tau$, and $n_\tau$ is the corresponding unimodular complex number. As $-u$ is the Green potential of a probability measure $\nu^+$ (see (\ref{eq:uphi})), $L$ is smooth, and $d\tau=in_\tau ds$ on $L$, it follows from (\ref{eq:normalonL}) and Gauss' theorem \cite[Thm. II.1.1]{SaffTotik} that
\begin{equation}
\label{eq:rightwinding}
2\pi i = 2\pi i \nu^+(D^+) = i \int_L\frac{\partial u}{\partial n}ds = 2\int_L\frac{\partial u}{\partial z}(\tau)d\tau.
\end{equation}
Furthermore, established for $L$, the chain of equalities above holds, in fact, for any Jordan curve contained in $\Omega_L$ and homologous to $L$ since $\partial u/\partial z$ is analytic on $\overline\Omega_L$, where we always take the outer normal with respect to the inner domain of that Jordan curve. Thus, a function
\begin{equation}
\label{eq:defPsi}
\Psi(z) := \exp\left\{2\int_1^z\frac{\partial u}{\partial z}(\tau)d\tau\right\}, \quad z\in\Omega_L,
\end{equation}
is well-defined and holomorphic in $\Omega_L$, where the integral is taken over any path joining $1$ and $z$ and lying entirely in $\Omega_L$.

Now, observe that the maximum principle for harmonic functions yields that $\partial u/\partial n$ is non-negative on each $\Gamma_\epsilon$ contained in $\Omega_L$. Set $\gamma_\epsilon(\tau)$ to be the subarc of $\Gamma_\epsilon$ obtained by traversing it into the counter-clockwise direction from $x_\epsilon$ to $\tau$, where $\{x_\epsilon\}:=\Gamma_\epsilon\cap(0,\infty)$. Then $\int_{\gamma_\epsilon(\tau)}(\partial u/\partial n)ds$ is a positive strictly increasing function of $\tau$ with the range $[0,2\pi]$ as $\tau$ winds once on $\Gamma_\epsilon$ in the positive direction starting from 1 by (\ref{eq:rightwinding}) and the remark right after. Hence, choosing the initial path of integration to be $\gamma(\tau):=[1,x_\epsilon]\cup\gamma_\epsilon(\tau)$, we derive as in (\ref{eq:rightwinding}) that
\begin{eqnarray}
\log|\Psi(\tau)| &=& \re\left(2\int_{\gamma(\tau)}\frac{\partial u}{\partial z}(\zeta)d\zeta\right)  = \re\left(2\int_1^{x_\epsilon}\frac{\partial u}{\partial z}(\zeta)d\zeta + i \int_{\gamma_\epsilon(\tau)}\frac{\partial u}{\partial n}ds\right) \nonumber \\
&=& 2\re\left(\int_1^{x_\epsilon}\frac{\partial u}{\partial z}(\zeta)d\zeta\right) =: \log\rho_\epsilon, \nonumber
\end{eqnarray}
and
\[
\Arg(\Psi(\tau)) = \im\left(2\int_{\gamma(\tau)}\frac{\partial u}{\partial z}(\zeta)d\zeta\right) = 2\im\left(\int_1^{x_\epsilon}\frac{\partial u}{\partial z}(\zeta)d\zeta\right) + \int_{\gamma_\epsilon(\tau)}\frac{\partial u}{\partial n}ds,
\]
$\tau\in\Gamma_\epsilon$, where $\Arg(\Psi)$ is the principal value of the argument of $\Psi$. Hence, $\Psi$ is univalent on each $\Gamma_\epsilon\subset\Omega_L$ and $\Psi(\Gamma_\epsilon)=\T_{\rho_\epsilon}$ (in particular, $\Psi(\Gamma)=\T$ and $\Psi(\pm1)=\pm1$). As each level line $\Gamma_\epsilon$ lies either entirely inside of $\Omega_L$ or entirely outside, $\Psi$ is a univalent function in the whole domain $\Omega_L$. It also holds that
\begin{equation}
\label{eq:recip}
\Psi(1/z)=1/\Psi(z), \quad z\in\Omega_L,
\end{equation}
as follows from a change of variables in (\ref{eq:defPsi}) and since
\[
(\partial u/\partial z)(1/z) = \pm H(\jt(1/z))\jt^\prime(1/z) = \mp z^2 H(\jt(z))\jt^\prime(z) = z^2(\partial u/\partial z)(z),
\]
$z\in \Omega_L\cap D^\pm$, by (\ref{eq:uhj}).

Set $f:=\jt\circ\Psi\circ\map$. Since $\map$ is univalent and holomorphic in $\jt(\Omega_L)\setminus F$, $\map(\jt(\Omega_L)\setminus F)=\Omega_L\cap D^-$, and $\Psi$ is injective and holomorphic in $\Omega_L$, $f$ is also univalent and holomorphic in $\jt(\Omega_L)\setminus F$. As $\map^\pm(F)=\Gamma^\pm$ and $\Psi(\Gamma^\pm)=\T^\pm$, it holds that $f^+$ and $f^-$ both map $F$ onto $[-1,1]$. It is also true that
\begin{eqnarray}
f^-(t) &=& \jt(\Psi(\map^-(t))) = \jt(\Psi(1/\map^+(t))) = \jt(1/\Psi(\map^+(t))) \nonumber \\
{}     &=& \jt(\Psi(\map^+(t))) = f^+(t), \quad t\in F,\nonumber
\end{eqnarray}
by (\ref{eq:boundary}) and (\ref{eq:recip}). Hence, $f$ is a holomorphic and univalent function in some $\jt(\Omega_L)$ that maps $F$ onto $[-1,1]$ and therefore the desired analytic parametrization of $F$ is given by $p=f^{-1}$.
\smallskip 
\newline
{\bf (c)$\Rightarrow$(a)}: Let $F$ be an analytic Jordan arc and $p$ be its holomorphic univalent parametrization. Denote, as usual, by $\Gamma$ the preimage of $F$ under the Joukovski transformation $\jt$. It was shown in Section \ref{subsec:aux_jouk} that there exists a function $\Psi$, holomorphic in some neighborhood of $\Gamma$, say $\Omega$, such that $\Psi(\Gamma)=\T$ and
\begin{equation}
\label{eq:reciprocal}
\Psi(1/z) = 1/\Psi(z), \quad z,1/z\in\Omega.
\end{equation}
In fact, $\Phi=\Psi^{-1}$ was constructed in Section \ref{subsec:aux_jouk}. Let $\rho\in(0,1)$ be such that 
\[
C_\rho(0),C_{1/\rho}(0)\subset\Psi(\Omega), \quad C_x(z_0) := \{z\in\C:~|z-z_0|=x\}, \quad x>0.
\]
Denote $\Gamma_\rho := \Phi(C_\rho(0))$ and $\Gamma_{1/\rho} := \Phi(C_{1/\rho}(0))$. It immediately follows from (\ref{eq:reciprocal}) that
\begin{equation}
\label{eq:reciprboundary}
\Gamma_{1/\rho} = \Gamma_\rho^{-1} = \{\tau\in\C:~1/\tau\in\Gamma_\rho\}.
\end{equation}
Denote by $\Omega_\rho$ the annular domain bounded by $\Gamma_\rho$ and $\Gamma_{1/\rho}$ and define
\[
u(z) := \log|\Psi(z)|, \quad z\in\Omega.
\]
Then $u$ is a harmonic function in some neighborhood of $\overline\Omega_\rho$ such that $u\equiv0$ on $\Gamma$, $u\equiv\log\rho$ on $\Gamma_\rho$, and $u\equiv-\log\rho$ on $\Gamma_{1/\rho}$. Furthermore, it follows right away from the definition of $u$ that
\begin{equation}
\label{eq:uderivative}
\frac{\partial u}{\partial z}(z) = \frac12\frac{\Psi^\prime(z)}{\Psi(z)}, \quad z\in\Omega.
\end{equation}
Let now $n_\tau$ stand for the unimodular complex number corresponding to the inner normal of $\Omega_\rho$ at $\tau\in\partial\Omega_\rho$. Since inner normals of $\partial\Psi(\Omega_\rho)$ are represented by $\xi/|\xi|$, $\xi\in C_\rho(0)$, and $-\xi/|\xi|$, $\xi\in C_{1/\rho}(0)$, we have by conformality at $\xi$ that
\begin{equation}
\label{eq:innernormals}
\pm\frac{\Psi(\tau)}{|\Psi(\tau)|} = \frac{\Psi^\prime(\tau)}{|\Psi^\prime(\tau)|} ~ n_\tau, \quad \tau\in\Gamma_{\rho^{\pm1}}.
\end{equation}
As $u$ is harmonic in a neighborhood of $\overline\Omega_\rho$, we deduce from (\ref{eq:uderivative}) and (\ref{eq:innernormals}) that
\begin{equation}
\label{eq:normalderivative}
\frac{\partial u}{\partial n}(\tau) = \langle\nabla u(\tau),\vec{n}_\tau\rangle = 2\re\left(\frac{\partial u}{\partial z}(\tau)n_\tau\right)  = \pm\rho^{\mp1}|\Psi^\prime(\tau)|, \quad \tau\in\Gamma_{\rho^{\pm1}},
\end{equation}
where $\vec{n}_\tau$ is the unit vector corresponding to the complex number $n_\tau$. Observe that
\[
\frac{\Psi^\prime(1/z)}{z^2} = \frac{\Psi^\prime(z)}{\Psi^2(z)}, \quad z,1/z\in\Omega,
\]
by (\ref{eq:reciprocal}). Hence, it follows from (\ref{eq:reciprboundary}) and (\ref{eq:normalderivative}) that
\begin{eqnarray}
\frac{\partial u}{\partial n}(1/\tau) &=& -\rho|\Psi^\prime(1/\tau)| = -\rho\left|\frac{\tau^2}{\Psi^2(\tau)}~\Psi^\prime(\tau)\right| = -\rho^{-1}|\tau|^2|\Psi^\prime(\tau)| \nonumber \\
\label{eq:normaldersym}
{} &=& -|\tau|^2\frac{\partial u}{\partial n}(\tau), \quad \tau\in\Gamma_\rho.
\end{eqnarray}

Now, we shall show that (\ref{eq:normaldersym}) implies the representation formula
\begin{equation}
\label{eq:urepresentation}
u(z) = \int_{\Gamma_\rho}\log\left|\frac{z-\tau}{1-z\tau}\right||\Psi^\prime(\tau)|\frac{ds}{2\pi\rho}, \quad z\in\Omega_\rho,
\end{equation}
where $ds$, as before, is the arclength differential. Before we do so, let us gather several facts that are easy consequences of the following version of the Green's formula. Let $O$ be a bounded open set with smooth boundary $\partial O$, and let $a$ and $b$ be two harmonic functions in some neighborhood of $\overline O$. Then
\begin{equation}
\label{eq:greensformula}
\int_{\partial O}a\frac{\partial b}{\partial n}ds = \int_{\partial O}b\frac{\partial a}{\partial n}ds,
\end{equation}
where $\partial/\partial n$ denotes differentiation in the direction of the inner normal of $O$. In what follows, all the partial derivatives are taken with respect to the inner normals of the corresponding domains.

Let $L$ be a smooth Jordan curve and $\Omega_L$ be the bounded component of its complement. If $a$ is a harmonic function on some neighborhood of $\overline\Omega_L$ then (\ref{eq:greensformula}) immediately implies that
\begin{equation}
\label{eq:fact1}
\int_L \frac{\partial a}{\partial n}ds = 0.
\end{equation}
Further, denote by $v_z$ the function $-\log|z-\cdot|$. Then
\begin{equation}
\label{eq:fact2}
\int_L \frac{\partial v_z}{\partial n}ds =  \left\{
\begin{array}{ll}
2\pi, & z\in\Omega_L, \\
0,    & z\notin\overline\Omega_L.
\end{array}
\right.
\end{equation}
Indeed, (\ref{eq:fact2}) follows from (\ref{eq:greensformula}) applied with $a\equiv1$ and $b=v_z$, by putting $O=\Omega_L$ in the second instance and $O=\Omega_L\setminus\overline D_z$ in the first, where is $\overline D_z$ is any closed disk around $z$ contained in $\Omega_L$.
Thus, (\ref{eq:greensformula}) and (\ref{eq:fact2}) yield that
\begin{eqnarray}
\int_{\partial\Omega_\rho}v_0\frac{\partial u}{\partial n}ds &=& \int_{\partial\Omega_\rho}u\frac{\partial v_0}{\partial n}ds = \log\rho\left(\int_{\Gamma_\rho}\frac{\partial v_0}{\partial n}ds - \int_{\Gamma_{1/\rho}}\frac{\partial v_0}{\partial n}ds\right) \nonumber \\
\label{eq:help1}
{} &=& -4\pi\log\rho.
\end{eqnarray}
Moreover, we deduce from (\ref{eq:normaldersym}) that
\begin{equation}
\label{eq:help2}
\int_{\Gamma_{1/\rho}} v_0\frac{\partial u}{\partial n}ds = \int_{\Gamma_\rho}v_0(1/\tau)\frac{\partial u}{\partial n}(1/\tau)\frac{ds}{|\tau|^2} = \int_{\Gamma_\rho}v_0\frac{\partial u}{\partial n}ds.
\end{equation}
So, we obtain from (\ref{eq:help1}) and (\ref{eq:help2}) that
\begin{equation}
\label{eq:fact3}
\int_{\Gamma_\rho}v_0\frac{\partial u}{\partial n}ds = \frac12\int_{\partial\Omega_\rho}v_0\frac{\partial u}{\partial n}ds = -2\pi\log\rho.
\end{equation}

Finally, let $O$ be $\Omega_\rho$ with the closed disk of radius $x$ around $z\in\Omega_\rho$ removed, $a=u$, and $b=v_z$. Then
\begin{eqnarray}
\int_{\partial O}u\frac{\partial v_z}{\partial n}ds &=& \log\rho\int_{\Gamma_\rho}\frac{\partial v_z}{\partial n}ds - \log\rho\int_{\Gamma_{1/\rho}}\frac{\partial v_z}{\partial n}ds -\frac1x\int_{C_x(z)}uds \nonumber \\
\label{eq:side1}
{} &=& -2\pi\log\rho -2\pi u(z)
\end{eqnarray}
by (\ref{eq:fact2}) and the mean-value property of harmonic functions. On the other hand, we get that
\begin{eqnarray}
\int_{\partial O}v_z\frac{\partial u}{\partial n}ds &=& \int_{\Gamma_\rho}v_z\frac{\partial u}{\partial n}ds + \int_{\Gamma_{1/\rho}}v_z\frac{\partial u}{\partial n}ds - \log x\int_{C_x(z)}\frac{\partial u}{\partial n}ds \nonumber \\
{} &=& \int_{\Gamma_\rho}(v_z(\tau)-v_z(1/\tau))\frac{\partial u}{\partial n}ds = \int_{\Gamma_\rho}\log\left|\frac{1-z\tau}{z-\tau}\right|\frac{\partial u}{\partial n}ds \nonumber \\
\label{eq:side2}
{} && + \int_{\Gamma_\rho}v_0\frac{\partial u}{\partial n}ds = \int_{\Gamma_\rho}\log\left|\frac{1-z\tau}{z-\tau}\right|\frac{\partial u}{\partial n}ds - 2\pi\log\rho,
\end{eqnarray}
where we used (\ref{eq:fact1}), (\ref{eq:normaldersym}), and (\ref{eq:fact3}). Thus, combining (\ref{eq:side1}) and (\ref{eq:side2}) with (\ref{eq:greensformula}) and (\ref{eq:normalderivative}), we get (\ref{eq:urepresentation}).

Now, we shall construct a triangular scheme such that $F$ is symmetric with respect to it. Obviously, this is equivalent to showing that there exists a sequence of sets $\{\hat E_n\}$, $\hat E_n=\{e_{j,n}\}_{j=1}^{2n}\subset D^+$, where $D^+$ is the interior of $\Gamma$, such that $|\hat r_n|=O(1)$ on $\Gamma$ and $|\hat r_n|=o(1)$ locally uniformly in $D^+$, where
\[
\hat r_n (z) = \prod_{e\in \hat E_n}\frac{z-e}{1-ez}.
\]
Set $\varrho:=\const|\Psi^\prime|$, $\int_{\Gamma_\rho}\varrho ds = 1$. Further, for each $n\in\N$, let $\{\Gamma_\rho^{j,n}\}_{j=1}^{2n}$ be a partition of $\Gamma_\rho$ into $2n$ simple  arcs that are pairwise disjoint except for possible common endpoints and satisfy
\[
\int_{\Gamma_\rho^{j,n}} \varrho ds=\frac{1}{2n}, \quad j=1,\ldots,2n.
\]
Clearly, it holds that
\begin{equation}
\label{eq:existEn1}
\diam(\Gamma_\rho^{j,n}) \leq \int_{\Gamma_\rho^{j,n}}ds \leq \frac{\const}{n}, \quad j=1,\ldots,2n,
\end{equation}
since $\inf_{\Gamma_\rho}\varrho>0$, where $\const$ is an absolute constant. It also holds that
\begin{equation}
\label{eq:existEn3}
0 = \int_{\Gamma_\rho} k_z(\tau)\varrho ds, \quad k_z(\tau):= \log\left|\frac{z-\tau}{1-z\tau}\right|,
\end{equation}
for any $z\in\Gamma$ by (\ref{eq:urepresentation}). Moreover, for each $z\in\Gamma$ and $\tau_1,\tau_2\in\Gamma_\rho$ we have that
\begin{equation}
\label{eq:existEn2}
|k_z(\tau_1)-k_z(\tau_2)| = \left|\log\left|1+\frac{(\tau_2-\tau_1)(1-z^2)}{(z-\tau_2)(1-z\tau_1)}\right|\right| \leq O(|\tau_2-\tau_1|),
\end{equation}
where the last bound does not depend on $z$.

Now, let $e_{j,n}$ be an arbitrary point belonging to  $\Gamma_\rho^{j,n}$, $j=1,\ldots,2n$, $n\in\N$. Then it follows from (\ref{eq:existEn1}) and (\ref{eq:existEn2}) that for any $z\in \Gamma$ and $\tau\in\Gamma_\rho^{j,n}$, $j=1,\ldots,2n$, we get
\begin{equation}
\label{eq:existEn4}
|k_z(\tau)-k_z(e_{j,n})| \leq O(\diam(\Gamma_\rho^{j,n})) \leq \frac{\const}{n},
\end{equation}
where $\const$ is, again, a constant independent of $z$, $j$, and $n$. Therefore, we see that the functions
\[
u_n(z) := \frac{1}{2n}\sum_{j=1}^{2n} k_z(e_{j,n}) = \sum_{j=1}^{2n} \int_{\Gamma_\rho^{j,n}} k_z(e_{j,n})\varrho ds
\]
are such that
\begin{eqnarray}
\|u_n\|_\Gamma &=& \left\|\sum\int_{\Gamma_\rho^{j,n}} \left( k_z(\tau) - k_z(e_{j,n}) \right) \varrho ds\right\|_{z\in\Gamma} \nonumber \\
{} &\leq& \sum \int_{\Gamma_\rho^{j,n}} \frac{\const}{n} \varrho ds = \frac{\const}{n} \nonumber
\end{eqnarray}
by (\ref{eq:existEn3}) and (\ref{eq:existEn4}). Finally, observe that $|\hat r_n| = \exp\{2n u_n\} = O(1)$ on $\Gamma$. This finishes the proof of the theorem since $\hat E_n\subset\Gamma_\rho\subset D^+$ and therefore $|\hat r_n|=o(1)$ locally uniformly in $D^+$ by the maximum principle and a normal family argument.
\boite

\section{Examples}
\label{sec:numerics}

Denote by $F_\alpha$, $\alpha\in\R$, the following set
\[
F_\alpha := \left\{\frac{i\alpha+x}{1+i\alpha x}:~ x\in[-1,1]\right\}, \;\;\; F_\alpha^{-1} := \left\{z:~ 1/z\in F_\alpha\right\}.
\]
Clearly, $F_\alpha$ is an analytic arc joining $-1$ and $1$ and therefore is symmetric with respect to some triangular scheme by Theorem \ref{thm:sp}. However, it can be easily computed that $|r^\pm(e;t)| \equiv 1$, $t\in F_\alpha$, if $e\in F_\alpha^{-1}\setminus\{\pm1\}$. Hence, if $\E=\{E_n\}$, $E_n=\{e_{j,n}\}_{j=1}^{2n}$, is such that $e_{j,n}\in F_\alpha^{-1}\setminus\{\pm1\}$ for all $n\in\N$ and $j\in\{1,\ldots,2n\}$, then $F$ will be symmetric with respect to $\E$. Further, it can be easily checked that for any $e\notin F_\alpha$, it holds that
\[
|r^+(e;t)r^+(e^*;t)| \equiv |r^-(e;t)r^-(e^*;t)| \equiv 1, \quad t\in F_\alpha,
\]
where
\[
e^* := \frac{2i\alpha+(1-\alpha^2)\bar e}{(1-\alpha^2)+2i\alpha\bar e}.
\]
Thus, if $\E=\{E_n\}$ is such that each $E_n$ contains $e$ and $e^*$ simultaneously, then $F$ will be symmetric with respect to $\E$. Finally, observe also that for $\alpha=0$ we get that $F_0=[-1,1]$ and $e^*=\bar e$.

Below, we plot zeros of polynomials orthogonal on $F_{-1/2}$ with respect to the weights $w_n(t) = e^tt^{-n}(t+4i/3)^{-n}$. In other words, these are the poles of the multipoint Pad\'e approximants to $f_\mu$, $d\mu(t)=ie^tdt/w^+(t)_{|F_{-1/2}}$, that corresponds to the interpolation scheme with half of the points at 0 and another half at $-4i/3$. The denominators of such approximants are constructed by solving orthogonality relations (\ref{eq:orthogonality}). Thus, finding the denominator of the approximants of degree $n$ amounts to solving system of linear equations whose coefficients are obtained from the moments of the measures $t^{-n}(t+4i/3)^{-n}d\mu(t)$. The computations were carried with MAPLE 9.5 software using 24 digits precision.

\begin{figure}[h!]
\centering
\includegraphics[scale=.4]{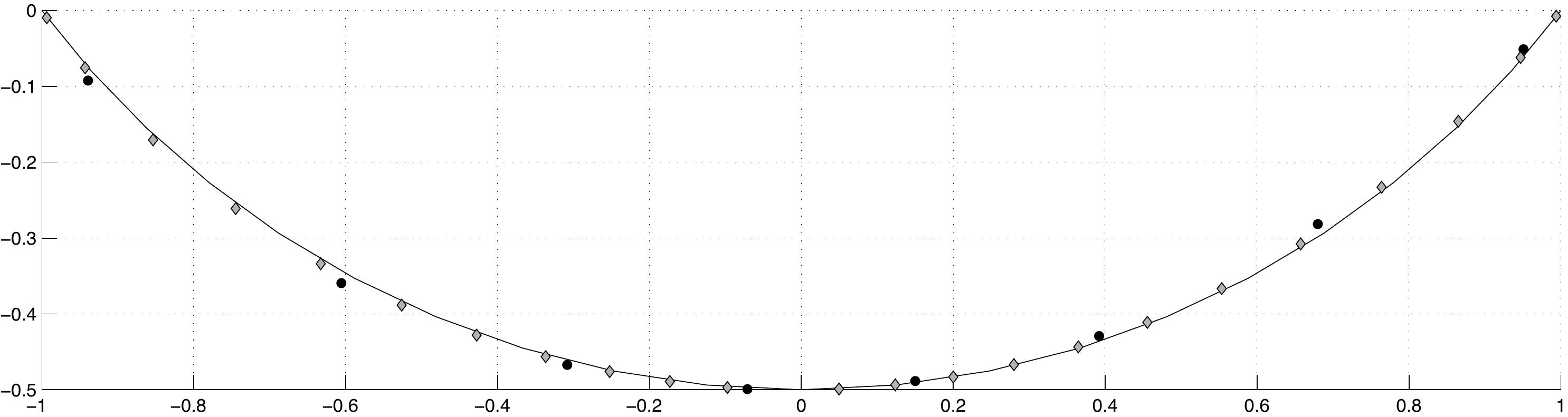}
\caption{\small Zeros of $q_8$ (disks) and $q_{24}$ (diamonds).}
\end{figure}

In the following example the contour $F$ is generated by $e_1:=(i-3)/4$, $e_2:=(87+6i)/104$, and $e_3:=-i/10$, in the sense that
\[
\left|\left(r(e_1;t)r(e_2;t)r(e_3;t)\right)^\pm\right| \equiv 1,
\]
and is computed numerically. Thus, $F$ is symmetric with respect to any triangular scheme such that $E_{3m}$ consists of $e_1$, $e_2$, and $e_3$ appearing $m$ times each, and $E_{3m+1}$ ($E_{3m+2}$) is obtained by adding to $E_{3m}$ an (two) arbitrary point (points) from $D$. Based on the derived discretizations of $F$, the moments of the measures $[(t-e_1)(t-e_2)(t-e_3)]^{-2m}h(t)dt$, $n=3m$, are computed for $n=24$ and $n=66$, where $h(t)=t$ if $\im(t)\geq0$ and $h(t)=\bar t$ otherwise, and the coefficients of the corresponding orthogonal polynomials are found by solving the respective linear systems. The computations were carried with MAPLE 9.5 software using 52 digits precision.

\begin{figure}[h!]
\centering
\includegraphics[scale=.4]{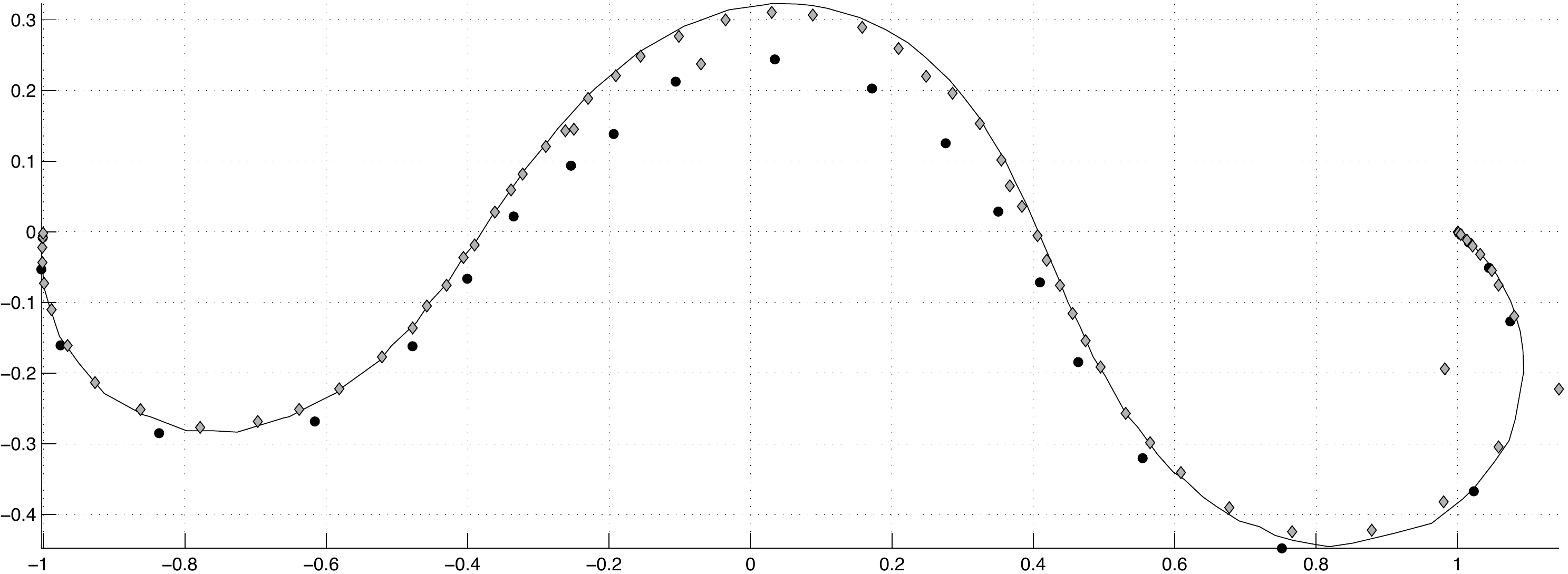}
\caption{\small Zeros of $q_{24}$ (disks) and $q_{66}$ (diamonds).}
\end{figure}

\bibliographystyle{plain}
\small
\bibliography{ci}

\end{document}

%% file: notation.tex
\newcommand{\boite}{\mbox{} \hfill \mbox{\rule{2mm}{2mm}}}
\newcommand{\T}		{\mathbb{T}}
\newcommand{\D}		{\mathbb{D}}

\newcommand{\R}		{\mathbb{R}}
\newcommand{\C}		{\mathbb{C}}
\newcommand{\N}		{\mathbb{N}}
\newcommand{\Z}		{\mathbb{Z}}
\newcommand{\op}	{\mathcal{P}}
\newcommand{\oq}	{\mathcal{Q}}
\newcommand{\oi}	{\mathcal{I}}

\newcommand{\oc}	{\mathcal{C}}
\newcommand{\os}	{\mathcal{S}}
\newcommand{\ot}	{\mathcal{T}}
\newcommand{\oh}	{\mathcal{H}}
\newcommand{\gm}	{G}
\newcommand{\szf}       {S}
\newcommand{\scf}       {c}
\newcommand{\dc}	{DC}
\newcommand{\cs}	{C(F)}
\newcommand{\map}	{\varphi}
\newcommand{\jt}	{J}
\newcommand{\Arg}	{\textnormal{Arg}}

\newcommand{\Log}	{\textnormal{Log}}

\newcommand{\ar}	{\theta}
\newcommand{\cp}	{\textnormal{cap}}
\newcommand{\E}		{\mathscr{E}}
\newcommand{\const}	{\textnormal{const.}}
\newcommand{\supp}	{\textnormal{supp}}

\newcommand{\diam}	{\textnormal{diam}}
\newcommand{\im}	{\textnormal{Im}}
\newcommand{\re}	{\textnormal{Re}}
\newcommand{\cws}	{\stackrel{*}{\to}}